\documentclass[11pt]{article}
 \usepackage{epsfig}
 \usepackage{amssymb,amsmath,amsthm,amscd}
\usepackage{latexsym}
\usepackage{graphicx}
\usepackage{subfigure}
\usepackage{color}

\usepackage{CJK}
 \theoremstyle{theorem}
 \newtheorem{thm}{Theorem}[section]
 \newtheorem{lem}[thm]{Lemma}
  \newtheorem{rem}[thm]{Remark}
 
 \newtheorem{prop}[thm]{Proposition}
 \newtheorem{cor}[thm]{Corollary}

\newcommand{\R}{\mathbb {R}}

 \title{``Hidden'' mechanisms for Gouy-Chapman layer and other critical features  via Poisson-Boltzmann equations}
  \author{Kaiying Huang\footnote{College of Mathematics, Sichuan University, Chengdu, China, 610064, P. R. China ({\tt huangky@scu.edu.cn}).}\; and  Weishi Liu\footnote{Department of Mathematics, University of Kansas,
1460 Jayhawk Blvd., Room 405,
Lawrence, Kansas 66045, USA ({\tt wsliu@ku.edu}).
}}
   \date{\empty}

\begin{document}
 \maketitle

 \begin{abstract}
In this work a dynamical system approach is taken to systematically investigate the one-dimensional classical Poisson-Boltzmann  (PB) equation with various boundary conditions. This framework, particularly, the phase space portrait,  has a unique advantage of a geometric view of the dynamical systems, which allows one to reveal and examine critical features of the PB models. More specifically, we are able  to reveal the mechanism of Gouy-Chapman layer:  the presence of an {\em equilibrium} for the PB equation, including equilibrium-at-infinity for Gouy-Chapman's original setup as the limiting case.
Several other critical, somehow counterintuitive, features revealed in this work are the saturation phenomenon of surface charge density, the uniform boundedness of electric pressure (given length) and of length (given electric pressure) in surface charge, and the critical length for  a reversal of electric force direction. All have a common mechanism: the presence of an equilibrium for the PB equation. 
We believe that the critical features presented from the classical PB models persist for modified PB systems up to a certain degree. On the other hand, any qualitative change in these features as the sophistication of the model is increasing is an indication of  new phenomena with new mechanisms.
\end{abstract}

\newpage

\tableofcontents

 \newpage
 \section{Introduction}\label{Intro}
  \setcounter{equation}{0}
Ion distribution in aqueous solutions plays a fundamental  role in  electrochemical applications and biological processes (see, e.g., \cite{i2,i3,i1}). The Poisson-Boltzmann (PB) type equations serve as basic models for studying and characterizing
the electric potential and equilibrium distribution of  ions around macromolecules  such as DNA,  protein or  colloid  in  electrolytes, and  have been successfully applied in many fields and led to  quite accurate predictions in experiments such as protein titration
behavior (see, e.g., \cite{re3,MAP2021,re2}).
\medskip

The classical PB equation is based on a continuum mean field-like approach. It neglects  the effects of the ionic correlations and needs improvements, particularly,  for   experimental phenomena at large voltage or high bulk concentrations (see, e.g., \cite{111,222}). Many modified PB type equations have been developed, by including various excess  potentials to account for effects such as finite ion sizes, entropy contribution of water dipoles, and dependence of the dielectric upon the electric potential or local ionic concentration  (\cite{33,44,88,55,66,77}). Despite their significance, our comprehension of PB type equations remains limited primarily due to the inherent nonlinearity (\cite{re3,pain30,re5,re4}). Various numerical methods, such as finite element methods, interface methods, and adaptive methods, have been developed to simulate PB type equations (see, for instance, the comprehensive review article \cite{pbN} and the references therein).
However the classical PB equation may  be the most widely used and effective electrostatic model for analysis of ionic solutions (see, e.g., \cite{qian2,c3,c4,i4,66}). A number of significant properties/phenomena were discovered via the classical PB models, notably the Gouy-Chapman layer (\cite{Cha13, Gouy10}), Debye screening (\cite{DH23}), etc.. On the other hand, even for the simplest classical PB equation, there is still ample room for improvement. 

In this work, we will revisit some well-known phenomena and relevant physical quantities in electrolytes via the classical PB equations, providing new insights from a perspective of dynamical systems. A great advantage of the dynamical system viewpoint is the intuitive/geometric  representation of typical solution behaviors by the phase portrait associated with the classical PB equation,
which proves to be extremely  beneficial. Specifically, it enables us to discover several crucial features presented in this paper and, to a certain extent, the discovery is by chance.  
%
Some highlights of our studies in terms of physical interpretations of analytical results are as follows:
\begin{itemize}
\item[(i)] Presence of equilibrium as  mechanisms for formation of Gouy-Chapman layer is revealed (cf Section \ref{layer});
\item[(ii)] The saturation phenomenon of surface charge density is recognized, and the asymptotic properties
of maximum surface charge density with respect to fixed charges are derived (cf Section \ref{secSCD});
\item[(iii)]
Interplays between the  electric pressure, surface charge density and separation distance are examined (cf Section \ref{sEP});
\item[(iv)]
The critical length required for direction reverse of electric force has been identified, and the effects of physical parameters on this critical length are discussed (cf Section \ref{critL}).
\end{itemize}

We believe
 the aforementioned findings are extremely important for PB theory and are hard to recognize without the view of phase portrait (global) dynamics of the PB equation.
\medskip


The rest of the paper is organized as follows. In Sections \ref{formulation} and \ref{PBphase}, we set up our problem and classify phase portraits of the planar system associated with the PB equation for all possible parameter ranges. In Section \ref{BVP}, we perform detailed analyses to enhance a quantitative understanding of the boundary value problems. Our main findings are presented in Section \ref{CF}, where the mechanisms for GC layer formation as well as  for critical features of several physical quantities are discussed.
A brief conclusion is drawn in Section \ref{concl}.


\section{Poisson-Boltzmann equations}\label{PBeqn}
 \setcounter{equation}{0}

\subsection{Setup of the system and boundary conditions}\label{formulation}

\begin{figure}[htpp]\label{fig33}
\centering
\includegraphics[width=0.5\textwidth]{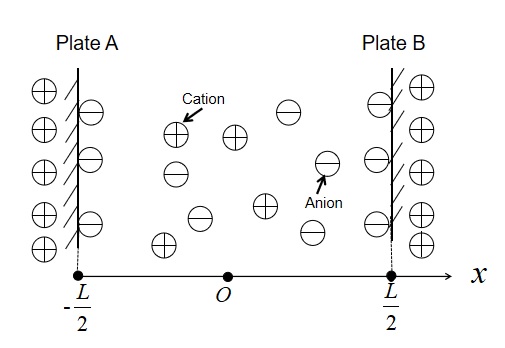}
\caption{
Schematic of a container along $x$ direction bounded by two parallel plates filled with an electrolyte solution.
}
\label{figf}
\end{figure}

Consider an ionic solution/mixture  with $N$ types of  ion species immersed in solvent  between two parallel plates of distance $L$ (see Figure \ref{figf}). The Poisson equation  for the electric potential $\varphi$ is given by
\begin{align}\label{gPBe}
-\frac{{\rm d}}{{\rm d}x}\left(\varepsilon_r(x)\varepsilon_0\frac{{\rm d}\varphi}{{\rm d}x}\right)=e_0\left(\sum_{j=1}^Nz_jc_j +q(x)\right),
\end{align}
where  $\varepsilon_r$ is  the relative
dielectric coefficient, $\varepsilon_0$ is the vacuum permittivity, $e_0$ is the elementary charge, $q(x)$ is the fixed charge, and, for $j=1,\cdots,N$,  $c_j$ and $z_j$ are the concentration and valence of the $j$-th ion species.

In equilibrium (no-flux) condition, for each $j$, the electrochemical potential $\mu_j=\mu_j^b$  is  constant with $\mu_j^b$ being the bulk chemical potential. The corresponding concentrations  $c_j=c_j(\varphi)$, $j=1,\cdots,N$, called the \emph{Boltzmann distributions}, are typically expressible in terms of the electric potential $\varphi$. In ideal ionic solutions for which  ions are treated as point-charges,  the  electrochemical potential $\mu_{j}$ is given by
 \begin{align}\label{uid}
\mu_{j}=\mu_j(\varphi,c_j)=z_{j} e_0 \varphi+k_B T \ln \frac{c_j}{C},
 \end{align}
  where $k_B$ is the Boltzmann constant, $T$ is the absolute temperature, and $C$ is a characteristic concentration. The corresponding Boltzmann distributions are
\begin{align}\label{dist}
c_j(\varphi)=c_j^be^{-\frac{z_je_0}{k_BT}\varphi},~j=1,\cdots,N,
\end{align}
where $c_j^b$ is the bulk concentration at where the electric potential $\varphi=0$.
Thus,  {\em in the equilibrium situation,} the Poisson equation (\ref{gPBe})  is then reduced to the classical PB  equation 
\begin{align}\label{cPB11}
-\frac{{\rm d}}{{\rm d}x}\left(\varepsilon_r(x)\varepsilon_0\frac{{\rm d}\varphi}{{\rm d}x}\right)=e_0\Big(\sum_{j=1}^Nz_jc_j^be^{-\frac{z_je_0}{k_BT}\varphi}+q(x)\Big).
\end{align}

For this work, we assume that the relative dielectric coefficient $\varepsilon_r(x)=\varepsilon_r$ and the fixed charge $q(x)=q_0$  with constants $\varepsilon_r$ and $q(x)=q_0$. Introduce the dimensionless variables
\begin{align}\label{sca1}
&\tilde{x}=\frac{x}{\lambda_D},~~\tilde L=\frac{L}{\lambda_D},~~\tilde{\varphi}=\frac{e_0}{k_BT}\varphi,~~{\tilde c_j^b}=\frac{c_j^b}{\sum_{i=1}^Nz_i^2c_i^b},\quad
\tilde q_0=\frac{q_0}{\sum_{i=1}^Nz_i^2c_i^b},
\end{align}
 where $\lambda_D$ is the Debye length given by
 \begin{align}\label{sca2}
 \lambda_D=\sqrt{\frac{\varepsilon_r\varepsilon_0k_BT}{e_0^2\sum_{i=1}^Nz_i^2c_i^b}}.
 \end{align}
The dimensionless version of PB (\ref{cPB11})  (after omitting the tildes) is given by
\begin{align}\label{PB}
\frac{{\rm d}^2\varphi}{{\rm d}x^2}+f(\varphi)=0\;\mbox{ where }\; f(\varphi)=\sum_{j=1}^{N} z_jc_j^be^{-z_j\varphi}+q_0.
\end{align}

 Associated to different setups of physical problems, two major types of boundary conditions are often considered:
 {\em Dirichlet boundary conditions}
\begin{align}\label{bc}
\varphi\left(-\frac{L}{2}\right)=\varphi_0\;\mbox{ and }\;\varphi\left(\frac{L}{2}\right)=\varphi_1;
\end{align}
and {\em  Neumann boundary conditions}
\begin{align}\label{bcc2}
\frac{{\rm d}\varphi}{{\rm d}x}\left(-\frac{L}{2}\right)=\sigma_0\;\mbox{ and }\;\frac{{\rm d}\varphi}{{\rm d}x}\left(\frac{L}{2}\right)=\sigma_1.
\end{align}

The existence and uniqueness of solutions of either the Dirichlet  boundary value problem (D-BVP) (\ref{PB}) and (\ref{bc}) or the Neumann boundary value problem (N-BVP) (\ref{PB}) and (\ref{bcc2}) can be established by classical theory such as  the standard theory of second-order ordinary differential equations and the upper-lower solution method. In this work we will apply dynamical system approach based heavily on  the phase plane portrait of the system associated with (\ref{PB}), and go much beyond the existence and uniqueness
to detect new features  and reveal underlying mechanisms, as presented in Section \ref{CF}.

We also mention that the unequal surface potentials ($\varphi_0\neq\varphi_1$) on the plates affect the monotonicity of the solution of the PB equation, see Section \ref{critL} for details. However, whether the charges on the plates are of equal magnitude or not does not have a fundamental impact on the analysis of our content of interest. For the sake of mathematical conciseness and to facilitate the comparison of our results with existing ones, without loss of generality, we only consider the symmetric Neumann boundary conditions
\begin{align}\label{nbc0}
\frac{{\rm d}\varphi}{{\rm d}x}\left(-\frac{L}{2}\right)=-\sigma,~~\frac{{\rm d}\varphi}{{\rm d}x}\left(\frac{L}{2}\right)=\sigma.
\end{align}
Here $\sigma>0$ (resp. $\sigma<0$) indicates that both two plates carry positive (resp. negative) charges.

Introducing the electric field intensity $u={{\rm d}\varphi}/{{\rm d}x}$,  the second-order PB equation (\ref{PB}) is rewritten
as a system of first-order equations, to be called {\em the PB system},
\begin{align} \label{Hpb}
\frac{{\rm d}\varphi}{{\rm d}x}= u,\quad \frac{{\rm d}u}{{\rm d}x}=-f(\varphi).
\end{align}
The PB system (\ref{Hpb}) is a Hamiltonian system (in fact, a Newtonian system)
 \begin{align*}
\frac{{\rm d}\varphi}{{\rm d}x}= \frac{\partial \mathcal{H}(\varphi,u)}{\partial u},\quad  \frac{{\rm d}u}{{\rm d}x}= -\frac{\partial \mathcal{H}(\varphi,u)}{\partial \varphi},
\end{align*}
with a Hamiltonian function (sum of the kinetic and potential energies)
\begin{align}\label{Hfun}
\mathcal{H}(\varphi,u)=\frac{1}{2}u^2+F(\varphi)\;\mbox{ where }\; F(\varphi)=\int f(\varphi)d\varphi=q_0\varphi-\sum_{j=1}^{N} c_j^be^{-z_j\varphi}.
\end{align}

For this specific problem, a special property of the Hamiltonian function ${\cal H}$ in (\ref{Hfun}) is that the potential $F(\varphi)$   is a {\em concave} function, that is,
\begin{align}\label{concaveF}
F''(\varphi)=-\sum_{j=1}^{N} z_j^2c_j^be^{-z_j\varphi}<0.
\end{align}
 The concavity of $F$ guarantees particularly the uniqueness of solutions of D-BVP (\ref{PB}) and (\ref{bc}) or the  N-BVP (\ref{PB}) and (\ref{bcc2}) (see Section \ref{BVP}).

\subsection{Phase portraits of PB systems}\label{PBphase}

\begin{figure}[htpp]
\centering
            \subfigure[Case i Saddle case]{
\includegraphics[width=0.3\textwidth]{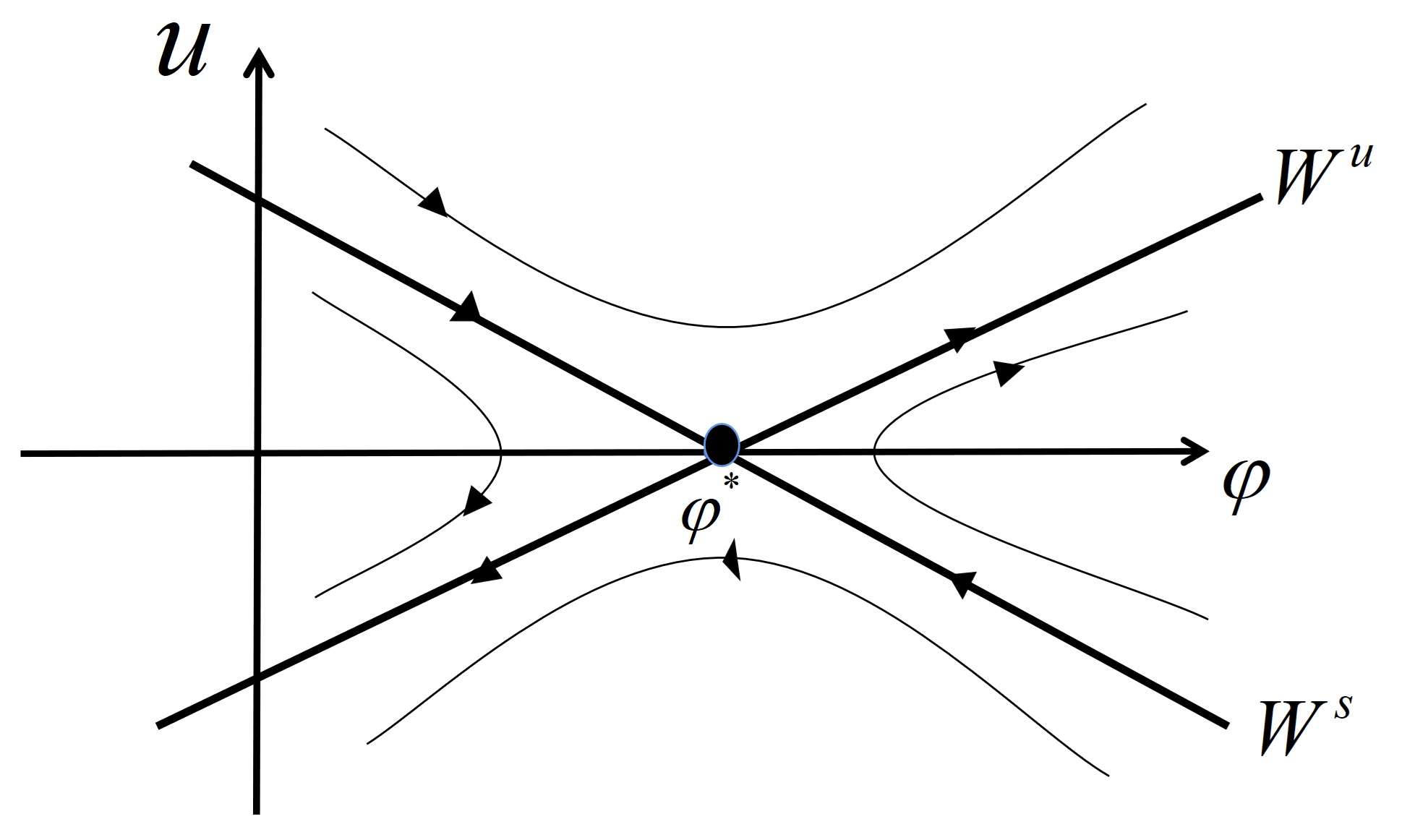}}
            \subfigure[Case ii.1 Equilibrium-at-infinity]{
\includegraphics[width=0.3\textwidth]{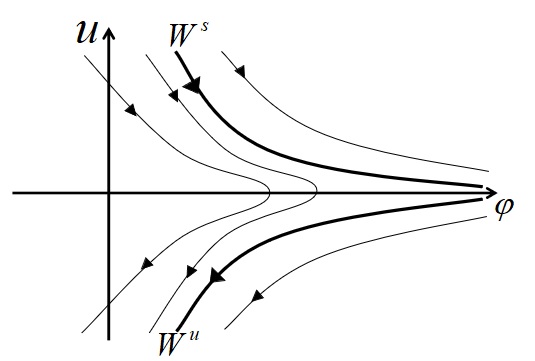}}
            \subfigure[Case ii.2 Equilibrium-at-infinity]{
\includegraphics[width=0.3\textwidth]{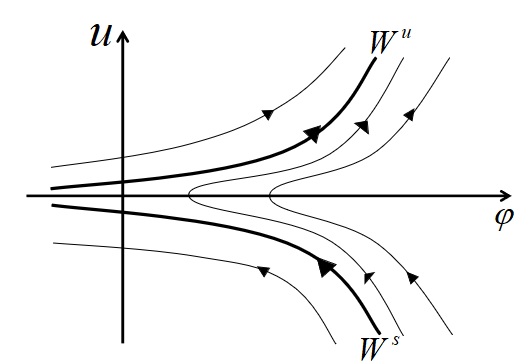}}\\
            \subfigure[Case iii.1 No equilibrium]{
\includegraphics[width=0.3\textwidth]{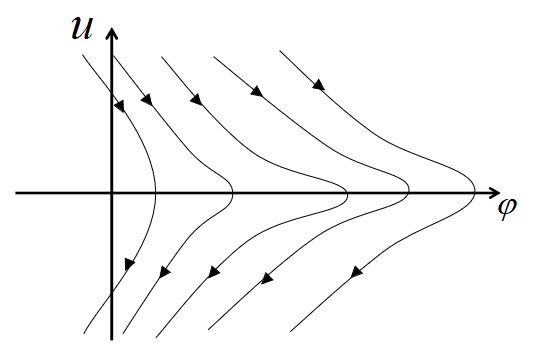}}
            \subfigure[Case iii.2 No equilibrium]{
\includegraphics[width=0.3\textwidth]{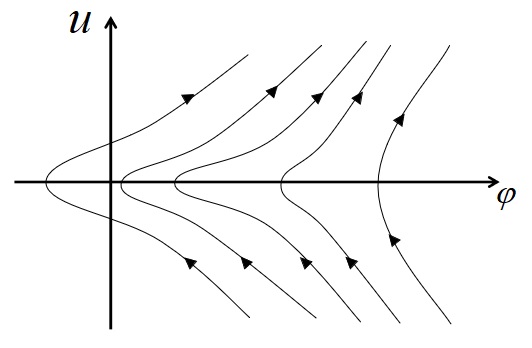}}
\caption{\em Phase portraits of system (\ref{Hpb}) for different cases.
}
\label{figp}
\end{figure}

As pointed out in the introduction, we will apply dynamical system theory to study  system (\ref{Hpb}) and the BVPs by starting with the phase plane portraits of system (\ref{Hpb}). We will often refer to the spatial variable $x$ as {\em time} for the dynamical system generated by system (\ref{Hpb}), for example, we will speak of the {\em time} taken for an orbit of system (\ref{Hpb}) moving from a point $A$ in the phase space to a point $B$.

The Hamiltonian (or Newtonian) structure of system (\ref{Hpb}) asserts that  the phase plane portrait is symmetric with respect to the $\varphi$-axis,  the level set ${\cal L}_h=\{(\varphi,u)\in\mathbb{R}^2|\mathcal{H}(\varphi,u)=h\}$ with level $h$ of the Hamiltonian function $\mathcal{H}(\varphi,u)$ is invariant, and a typical level set  is  a  one-dimensional curve. These properties allow one to get a \emph{complete} classification of phase plane portraits of systems (\ref{Hpb}) for all cases. For an important reason to be  easily seen later, we will call $E_{\infty}=(\infty,0)$ (resp.  $E_{-\infty}=(-\infty,0)$) an {\em equilibrium-at-infinity}  for system (\ref{Hpb}) if $f(\infty)=0$ (resp. $f(-\infty)=0$).
\begin{prop}\label{th1}
System (\ref{Hpb}) has five  distinct phase portraits detailed below.
\begin{itemize}
\item[{\rm(i)}] {\bf [Saddle]}  If, either $z_iz_j<0$ for some $i,j\in\{1,\cdots,N\}$ or $z_jq_0<0$ for all $j\in\{1,\cdots,N\}$, then system (\ref{Hpb}) has exactly one saddle equilibrium $(\varphi^*,0)$ with  $\varphi^*\in (-\infty,\infty)$ (see Figure \ref{figp}(a)).
\item[{\rm(ii)}] {\bf [Equilibrium-at-infinity]} There are two similar situations:

{\rm(ii.1)} If $q_0=0$ and $z_j>0$ for all $j=1,\cdots,N$, then system (\ref{Hpb}) has an equilibrium-at-infinity  $E_{\infty}=(\infty,0)$. The level set
${\cal L}_0$ with the level $h=0$ is 
\[{\cal L}_0=E_{\infty}\cup W^s\cup W^u,\]
where $W^s$ and $W^u$ are two curves (bold curves in Figure \ref{figp}(b))  given by
\[W^s=\{(\varphi,u): u=\sqrt{-2F(\varphi)}\} \;\mbox{ and }\; W^u=\{(\varphi,u): u=-\sqrt{-2F(\varphi)}\}.\]

The phase plane is separated by  ${\cal L}_0$ into  two regions:  one is bounded by ${\cal L}_0$ and  consists of level curves ${\cal L}_h$ with $h<0$ that intersect with the $\varphi$-axis; the other region consists of level curves ${\cal L}_h$ with $h>0$ that have no intersection with the $\varphi$-axis, as sketched in Figure \ref{figp}(b).

{\rm(ii.2)} If $q_0=0$ and $z_j<0$ for all $j=1,\cdots,N$, then system (\ref{Hpb}) has an equilibrium-at-infinity $E_{-\infty}=(-\infty,0)$.
The level set
${\cal L}_0$ with the level $h=0$ is 
\[{\cal L}_0=E_{-\infty}\cup W^s\cup W^u,\]
where   $W^s$ and $W^u$ are two curves (bold curves in Figure \ref{figp}(c)) given by
\[W^s=\{(\varphi,u): u=\sqrt{-2F(\varphi)}\} \;\mbox{ and }\; W^u=\{(\varphi,u): u=-\sqrt{-2F(\varphi)}\}.\]
The phase space is separated by  ${\cal L}_0$ into  two regions:  one is bounded by ${\cal L}_0$ and  consists of level curves ${\cal L}_h$ with $h<0$ that intersect with the $\varphi$-axis and the other region consists of level curves ${\cal L}_h$ with $h>0$ that have no intersection with the $\varphi$-axis, as sketched in Figure \ref{figp}(c).

 \item[{\rm(iii)}] {\bf [No equilibrium]} There are two similar situations:

{\rm(iii.1)} If $q_0>0$ and $z_j>0$ for all $j=1,\cdots,N$, then system (\ref{Hpb}) has no  equilibrium. Every level curve ${\cal L}_h$ is bounded from left and intersects with the $\varphi$-axis,  as sketched in Figure \ref{figp}(d).

{\rm(iii.2)} If $q_0<0$ and $z_j<0$ for all $j=1,\cdots,N$, then system (\ref{Hpb}) has no  equilibrium.  Every level curve ${\cal L}_h$ is bounded from right and intersects with the $\varphi$-axis,  as sketched in Figure \ref{figp}(e).
\end{itemize}
\end{prop}
\begin{proof} (i) It follows from $f'(\varphi)=-\sum_{j=1}^{N} z_j^2c_{j}^{b}e^{-z_j\varphi}<0$ that $f(\varphi)$ is decreasing in $\varphi$. Note that
\begin{align*}
\lim_{\varphi\rightarrow -\infty}f(\varphi)=\left\{
\begin{aligned}
 \infty, &~~\textit{if $z_j>0$ for some $j\in\{1,\cdots,N\}$},\\
 q_0,&~~\textit{if $z_j<0$ for all $j\in\{1,\cdots,N\}$},
\end{aligned}
\right.
\end{align*}
and
\begin{align*}
\lim_{\varphi\rightarrow \infty}f(\varphi)=\left\{
\begin{aligned}
-\infty, &~~\textit{if $z_j<0$ for some $j\in\{1,\cdots,N\}$},\\
q_0, &~~\textit{if $z_j>0$ for all $j\in\{1,\cdots,N\}$}.
\end{aligned}
\right.
\end{align*}
We then infer that $f(\varphi)=0$ has a unique root $\varphi=\varphi^*\in (-\infty,\infty)$ and system (\ref{Hpb}) has a unique equilibrium
$E=(\varphi^*,0)$.  The eigenvalues of the linearization of system (\ref{Hpb}) at $E$ are
\[\lambda_{1,2}=\pm\sqrt{\sum_{j=1}^{N} z_j^2c_{j}^{b}e^{-z_j\varphi^*}}.\]
Thus $E$ is a saddle. Any level set
${\cal L}_h=\{(\varphi,u)\in\mathbb{R}^2|\mathcal{H}(\varphi,u)=h\}$ of $\mathcal{H}$ is invariant. For typical $h$, the  level set
${\cal L}_h$ is  the union of graphs of two functions
\[u=\pm\sqrt{2(h-F(\varphi))}.\]
In particular, the stable and unstable manifolds (bold curves in Figure \ref{figp}(a)) lie on the same level set as that of $E$ with level
$h=\mathcal{H}(\varphi^*,0)=F(\varphi^*)$. The phase plane portrait of system (\ref{Hpb}) can be sketched easily, as shown in Figure \ref{figp}(a).

(ii) Take, for example, case (ii.1). It is clear that $f(\infty)=0$ so system (\ref{Hpb}) has an equilibrium-at-infinity $E_{\infty}=(\infty,0)$.
The two curves
\[W^s=\{(\varphi,u): u=\sqrt{-2F(\varphi)}\} \;\mbox{ and }\; W^u=\{(\varphi,u): u=-\sqrt{-2F(\varphi)}\}\]
lie on the same level curve ${\cal L}_0$ as that of $E_{\infty}$ and, are the stable and unstable manifolds of $E_{\infty}$ in the sense that the $\omega$-limit set
$\omega(\varphi,u)=\{(\infty,0)\}$ for $(\varphi,u)\in W^s$,  and the $\alpha$-limit set $\alpha(\varphi,u)=\{(\infty,0)\}$ for $(\varphi,u)\in W^u$ due to that
\[u=\pm\sqrt{-2F(\varphi)}\to 0\;\mbox{ as }\;\varphi\to \infty.\]
 Note that, $F(\varphi)<0$ for all $\varphi$. Thus, for $h<0$, the level curve ${\cal L}_h=\{u^2=2(h-F(\varphi))\}$ has a unique intersection with the $\varphi$-axis and is bounded by $W^s$ and $W^u$, and for $h>0$, the level curve ${\cal L}_h$ has no intersection with the $\varphi$-axis and lies above $W^s$ or below $W^u$.

 (iii) In this case, one has $|f(\varphi)|\ge |q_0|$ for $-\infty\le \varphi\le \infty$. Hence, system (\ref{Hpb}) has no equilibrium. Also  the image of $F(\varphi)$ is $\R$ which implies that  every level curve intersects with the $\varphi$-axis.
\end{proof}


The classification of the phase portraits of system (\ref{Hpb}) offers great convenience for investigating BVPs, as the phase portraits vividly illustrate the distribution and qualitative information of the solution trajectories of BVP.


\subsection{BVPs of the PB system (\ref{PB})} \label{BVP}
\subsubsection{Dirichlet boundary conditions} We will first examine existence, uniqueness and several properties  
for solutions of D-BVP (\ref{PB})-(\ref{bc}) with $\varphi_0=\varphi_1$. We mention that one can establish the results for $\varphi_0\neq\varphi_1$ in a similar way (see Section \ref{critL} for details).

\begin{thm}\label{thequ}
Assume  $\varphi_0=\varphi_1$. The D-BVP (\ref{PB})-(\ref{bc}) has a unique solution.
\end{thm}

\begin{figure}[htpp]
\centering
\includegraphics[width=0.5\textwidth]{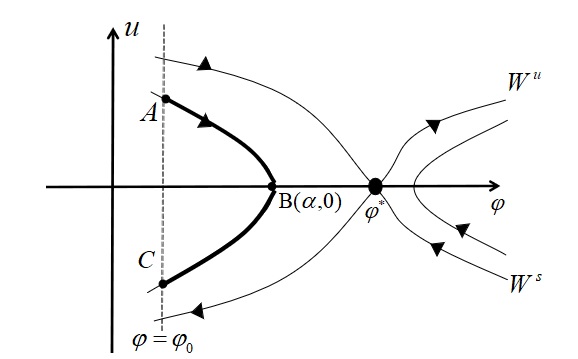}
\caption{
Saddle case with $\varphi_0<\varphi^*$:  an orbit of D-BVP
represented by the arc ABC.
}
\label{equ}
\end{figure}
\begin{proof} Depending on different phase plane portraits, we investigate the following cases to substantiate our conclusions.

(i) We first consider the saddle case with $\varphi^*\in (-\infty,\infty)$. For  $\varphi_0<\varphi^*$, consider the solution  $\varphi(x)$ of
D-BVP (\ref{PB}) with (\ref{bc}), whose orbit is represented by the arc ABC in Figure \ref{equ}.
System (\ref{Hpb}) admits the \emph{time-reversibility}, that is, if $(\varphi(x),u(x))$ satisfies (\ref{Hpb}),
then $(\varphi(-x),-u(-x))$ also satisfies (\ref{Hpb}), which implies the \emph{time} for the orbit taken from $A$ to $B$ is equal to the \emph{time} from $B$ to $C$. Then we see that $u(x)=\varphi'(x)>0$ for $x\in(-L/2,0)$ and $u(0)=\varphi'(0)=0$. Set $\alpha=\varphi(0)\in(\varphi_0,\varphi^*)$. Using the first integral $\mathcal{H}$, we have
\[\frac{u^2(x)}{2}+F(\varphi(x))=\mathcal{H}(\varphi(x),u(x))=\mathcal{H}(\alpha,0)=F(\alpha),\]
which leads to, for $x\in(-\frac{L}{2},0)$,
\[\frac{{\rm d}\varphi}{{\rm d}x}=\sqrt{2(F(\alpha)-F(\varphi))},~~\textit{or equivalently,}~~\frac{{\rm d}\varphi}{\sqrt{2(F(\alpha)-F(\varphi))}}={\rm d}x.\]
Integrating from $x=-L/2$ to $x=0$, one gets the  \emph{time} from $A$ to $B$
\[\int_{\varphi_0}^\alpha\frac{{\rm d}\varphi}{\sqrt{2(F(\alpha)-F(\varphi))}}=\frac{L}{2}.\]
For $\alpha\in(\varphi_0,\varphi^*)$, denote 
\begin{align}\label{time1}
T(\alpha):=\sqrt{2}\int_{\varphi_0}^\alpha\frac{{\rm d}\varphi}{\sqrt{F(\alpha)-F(\varphi)}}.
\end{align}
The function $T(\alpha)$ will be called the time-map. One has that the solutions of D-BVP (\ref{PB})-(\ref{bc})  is in one-to-one correspondence with the solutions $\alpha$'s of $T(\alpha)=L$. 

Next we claim that
\begin{align}\label{TM}
\lim_{\alpha\rightarrow \varphi_0^+}T(\alpha)=0,\quad \lim_{\alpha\rightarrow\varphi^{*-}}T(\alpha)=\infty,\quad \frac{{\rm d}T}{{\rm d}\alpha}>0.
\end{align}
As we commented earlier that $F$ is concave (see (\ref{concaveF})), or equivalently,  that $f(\varphi)$ is monotonically decreasing in $\varphi\in(\varphi_0,\varphi^*)$ and $f(\varphi^*)=0$. Thus,
\[F(\alpha)-F(\varphi)=\int_{\varphi}^\alpha f(s){\rm d}s>f(\alpha)(\alpha-\varphi)~~\textit{for}~\varphi_0<\varphi<\alpha<\varphi^*,\]
and
\[T(\alpha)<\sqrt{2}\int_{\varphi_0}^\alpha\frac{{\rm d}\varphi}{\sqrt{ f(\alpha)(\alpha-\varphi)}}=2^{\frac{3}{2}}\sqrt{\frac{\alpha-\varphi_0}{f(\alpha)}},\]
which implies $\lim_{\alpha\rightarrow\varphi_0^+}T(\alpha)=0$. Set $m=-f'(\varphi^*)>0$.
Then there exists $\delta>0$ such that for $\varphi\in(\varphi^*-\delta,\varphi^*)$ we have $f(\varphi)<2m(\varphi^*-\varphi)$. Similar to above for $\varphi^*-\delta<\varphi<\alpha<\varphi^*$ we have
\[F(\alpha)-F(\varphi)=\int_{\varphi}^\alpha f(s){\rm d}s<m((\varphi^*-\varphi)^2-(\varphi^*-\alpha)^2),\]
and
\[T(\alpha)>\sqrt{2}\int_{\varphi^*-\delta}^\alpha\frac{{\rm d}\varphi}{\sqrt{m((\varphi^*-\varphi)^2-(\varphi^*-\alpha)^2)}}=\sqrt{\frac{2}{m}}\ln\frac{\delta+\sqrt{\delta^2-(\varphi^*-\alpha)^2}}{\varphi^*-\alpha},\]
which implies $\lim_{\alpha\to\varphi^{*-}}T(\alpha)=\infty$. It should be pointed out that two limitations
$\lim_{\alpha\rightarrow\varphi_0^+}T(\alpha)=0,~\lim_{\alpha\rightarrow\varphi^{*-}}T(\alpha)=\infty$ can also be derived directly
from the phase portrait of the planar differential system (\ref{Hpb}). Finally, making the change of variable $\varphi=(\alpha-\varphi_0)t+\varphi_0$, one can transform $T(\alpha)$ into
\begin{align}\label{time2}
T(\alpha)=\sqrt{2}\int_{0}^1\frac{\alpha-\varphi_0}{\sqrt{F(\alpha)-F((\alpha-\varphi_0)t+\varphi_0)}}{\rm d}t,
\end{align}
which leads to
\begin{align}\label{Tdot}\begin{split}
\frac{{\rm d}T}{{\rm d}\alpha}=&\frac{\sqrt{2}}{2}\int_{0}^1\frac{H(\alpha)-H((\alpha-\varphi_0)t+\varphi_0)}{
(F(\alpha)-F((\alpha-\varphi_0)t+\varphi_0))^{3/2}}{\rm d}t\\
=&\frac{\sqrt{2}}{2}\int_{\varphi_0}^{\alpha}\frac{H(\alpha)-H(\varphi)}{
(F(\alpha)-F(\varphi))^{3/2}}{\rm d}\varphi,
\end{split}
\end{align}
where $H(x)=2F(x)-(x-\varphi_0)f(x)$. It follows from $f(\varphi^*)=0$ and $f'(\varphi)<0$ that $H'(x)=f(x)-(x-\varphi_0)f'(x)$ is positive for $x\in(\varphi_0,\varphi^*)$. Therefore $H(\alpha)-H((\alpha-\varphi_0)t+\varphi_0)$ is positive for $t\in(0,1)$, and
${\rm d}T/{\rm d}\alpha$ is also positive for $\alpha\in(\varphi_0,\varphi^*)$. It follows from (\ref{TM}) that BVP with $\varphi_0<\varphi^*$ has a unique
solution.

For the saddle case with $\varphi_0=\varphi^*$, its proof is trivial.
For the saddle case with $\varphi_0>\varphi^*$,  we can construct the  time-map
 \begin{align}
\hat T(\alpha):=2\int_{\alpha}^{\varphi_1}\frac{{\rm d}\varphi}{\sqrt{2(F(\alpha)-F(\varphi))}},~~~\alpha\in(\varphi^*,\varphi_1).
\end{align}
Using similar arguments as described above, we obtain
\[\lim_{\alpha\rightarrow\varphi^{*+}}\hat T(\alpha)=\infty,~\lim_{\alpha\rightarrow\varphi_1^-}\hat T(\alpha)=0,~~\frac{{\rm d}\hat T}{{\rm d}\alpha}<0,\]
which results in the statement.

(ii) For the equilibrium-at-infinity case (ii.1) in Proposition  \ref{th1},  we  consider the time-map $T(\alpha)$ for $\varphi_0<\alpha<\infty$. Clearly, we also have $
\lim_{\alpha\rightarrow\varphi_0^+}T(\alpha)=0,~\frac{{\rm d}T}{{\rm d}\alpha}>0. $ It suffices to 
show that $\lim_{\alpha\rightarrow\infty}T(\alpha)=\infty$. Note that $\lim_{\varphi\rightarrow\infty}f(\varphi)=0$, and there exist two positive constants $M,C$ such that $0<f(\varphi)<C$ for $\varphi>M$. Therefore, for $\alpha>\varphi>M$,
\[F(\alpha)-F(\varphi)=\int_{\varphi}^\alpha f(s){\rm d}s<C(\alpha-\varphi),\]
and hence,
\[T(\alpha)>\sqrt{2}\int_{M}^\alpha\frac{{\rm d}\varphi}{\sqrt{C(\alpha-\varphi)}}=2^{\frac{3}{2}}\sqrt{\frac{\alpha-M}{C}},\]
which implies $\lim_{\alpha\rightarrow\infty}T(\alpha)=\infty$. We thus complete the proof of statement (ii.1).

For the equilibrium-at-infinity case (ii.2) in Proposition \ref{th1},  one should consider the time-map $\hat T(\alpha)$ for $-\infty<\alpha<\varphi_0$ and show that
\[\lim_{\alpha\rightarrow-\infty}\hat T(\alpha)=\infty,~\lim_{\alpha\rightarrow\varphi_0^-}\hat T(\alpha)=0,~~\frac{{\rm d}\hat T}{{\rm d}\alpha}<0.\]

 (iii) Finally, the cases  (iii.1) and (iii.2) in Proposition \ref{th1} can be readily demonstrated by employing the identical reasoning as presented earlier for the equilibrium-at-infinity scenarios (ii.1) and (ii.2).
\end{proof}

\begin{rem}\label{monoTM}
We emphasize that it is the concavity property (\ref{concaveF}) of the potential function $F$ that gives the monotonicity of
the time-map $T(\alpha)$ defined in (\ref{time1}) for the PB system (\ref{PB}), which, in turn,
yields  the uniqueness  of solutions of BVPs. In general,   time-maps $T(\alpha)$ for nonlinear models
may not be monotonic, and multiple solutions of BVPs could exist.
\end{rem}

Using the phase space portraits of the PB system  (\ref{Hpb}), we can readily obtain the following observations in an intuitive manner: 
the solution $\varphi(x)$ of D-BVP (\ref{PB})-(\ref{bc}) with $\varphi_0=\varphi_1$, corresponding to each case in Proposition \ref{th1}, satisfies the following property.
\begin{itemize}
\item[{\rm (i)}] {\bf [{Saddle}]}

{\rm(i.1)} If $\varphi_0<\varphi^*$, then the electric field $u(x)=\varphi'(x)$ is decreasing in $x$.

{\rm(i.2)} If $\varphi_0=\varphi^*$, then the electric field $u(x)=\varphi'(x)\equiv 0$.

{\rm(i.3)} If $\varphi_0>\varphi^*$, then the electric field $u(x)=\varphi'(x)$ is increasing in $x$.

\item[{\rm (ii)}] {\bf [Equilibrium-at-infinity]}

{\rm(ii.1)} For $E_{\infty}$ case: The electric field  $u(x)=\varphi'(x)$ is decreasing in $x$.

{\rm(ii.2)} For $E_{-\infty}$ case:  The electric field  $u(x)=\varphi'(x)$ is increasing in $x$.

\item[{\rm (iii)}] {\bf [No equilibrium]}

{\rm(iii.1)} The electric field $u(x)=\varphi'(x)$ is  decreasing in $x$.

{\rm(iii.2)} The electric field  $u(x)=\varphi'(x)$ is  increasing in $x$.
\end{itemize}

%
%
%
%
%
%
%
%
%
%
%

\subsubsection{Neumann boundary conditions} Next we are concerned with the symmetric N-BVP (\ref{PB}) and (\ref{nbc0}). Without loss of generality, assume $\sigma>0$. (Otherwise one can consider $\tilde\varphi(x)=\varphi(-x)$.) In this case, the phase portrait of (\ref{Hpb}) must belong to one of cases (i), (ii.2) and (iii.2) in Proposition \ref{th1} if  the N-BVP (\ref{PB}) and (\ref{nbc0}) has a solution.

For easy of notation, denote $\varphi_e\in\{\varphi^*,~\infty,-\infty\}$  presenting the equilibrium in each cases. Let $F_e=F(\varphi_e)$, that is, 
$F_e=q_0\varphi^*-\sum_{j=1}^Nc_j^be^{-z_j\varphi^*}$ in the saddle case, and
$F_e=0$ in the equilibrium-at-infinity case. 

 An examination of  the phase portraits of PB systems gives that the Neumann boundary condition (\ref{nbc0}) corresponds to the Dirichlet boundary condition   $\varphi(-L/2)=\varphi(L/2):=\varphi_s$ via
\begin{align}\label{re1}
\frac{\sigma^2}{2}+F(\varphi_s)=F(\alpha),~~\varphi_s>\alpha=\varphi(0).
\end{align}
\begin{lem}\label{funB}
There exists a unique  function $\varphi_s=G(\alpha,\sigma)$ satisfying (\ref{re1}).
\end{lem}
Then, by the proof of Theorem \ref{thequ}, we see that  solutions of BVP (\ref{PB}) and (\ref{nbc0}) is in one-to-one correspondence with solutions $\alpha=\alpha(\sigma,L)$ of 
\begin{align}\label{re2}
M(\alpha,\sigma)=\frac{L}{\sqrt{2}},~~{\rm where}~~M(\alpha,\sigma):=\int_{\alpha}^{G(\alpha,\sigma)}\frac{{\rm d}\varphi}{\sqrt{F(\alpha)-F(\varphi)}}.
\end{align}

\begin{figure}[htpp]
\centering
            \subfigure[case~(I)~with $\varphi_e=\varphi^*$]{
\includegraphics[width=0.32\textwidth]{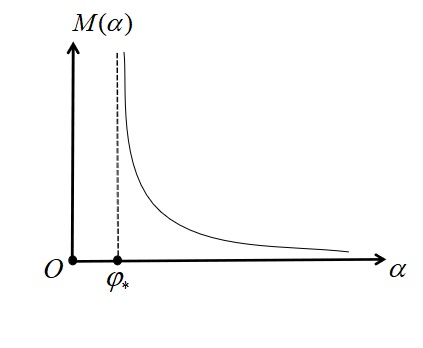}}
            \subfigure[case~(I)~with $\varphi_e=-\infty$]{
\includegraphics[width=0.32\textwidth]{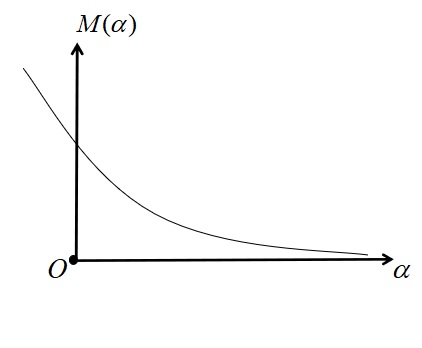}}
            \subfigure[case~(II)]{
\includegraphics[width=0.32\textwidth]{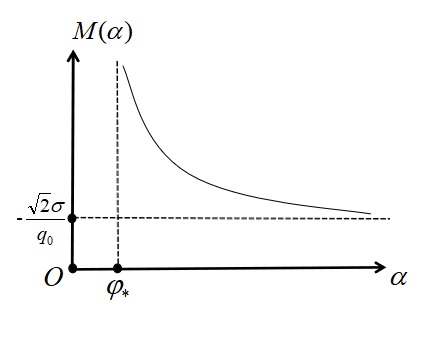}}
            \subfigure[case~(III)]{
\includegraphics[width=0.32\textwidth]{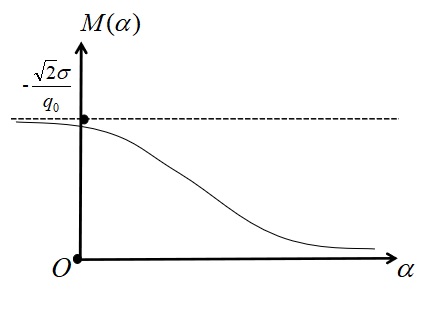}}
\\
\caption{
Sketch of the graph of the function $M(\alpha)$ for a fixed $\sigma$.
}
\label{Ma}
\end{figure}

\begin{thm}\label{thN}
Assume  $\sigma>0$. The N-BVP (\ref{PB}) and (\ref{nbc0}) has a unique solution $\varphi(x)$ if and only if one of the following three cases hold:
\begin{itemize}
\item[{\rm(I)}] either $z_iz_j<0$ for some $i,j\in\{1,\cdots,N\}$ or $z_i<0$ for all $i\in\{1,\cdots,N\}$ and $q_0\geq0$;

\item[{\rm(II)}] $z_j>0$ for all $j\in\{1,\cdots,N\}$, $q_0<0$ and
$$
(\sigma,L)\in \Omega_1=\Big\{(\sigma,L):L>-\frac{2\sigma}{q_0}\Big\}=\Big\{(\sigma,L):2\sigma+{q_0}L<0\Big\};
$$

\item[{\rm(III)}] $z_j<0$ for all $j\in\{1,\cdots,N\}$, $q_0<0$ and
$$
(\sigma,L)\in \Omega_2=\Big\{(\sigma,L):L<-\frac{2\sigma}{q_0}\Big\}=\Big\{(\sigma,L):2\sigma+{q_0}L>0\Big\}.
$$
\end{itemize}
\end{thm}

\begin{proof}
 It is noted that case (I) indicates the presence of an equilibrium point $\varphi_e$  and the existence of an anion in the electrolyte. For case (II), it is stated that an equilibrium point exists, but there is an absence of an anion in the electrolyte. Lastly, case (III) asserts that there is no equilibrium point of system (\ref{Hpb}).

Owing to the Implicit Function Theorem and Mean Value Theorem, we deduce that the function $G(\alpha,\sigma)$ satisfies
 $$
\frac{\partial G}{\partial \alpha}=\frac{f(\alpha)}{f(G)}\in(0,1),~~\frac{\partial G}{\partial \sigma}=-\frac{\sigma}{f(G)}>0,~\lim_{\sigma\to\infty}G=\lim_{\alpha\to\infty}G=\infty,
$$
and
\begin{align}\label{B}
0<-\frac{\sigma^2}{2f(G)}<G-\alpha<-\frac{\sigma^2}{2f(\alpha)}.
\end{align}
 We fix $\sigma$  and regard $M=M(\alpha)$ as a function in $\alpha$. Then, taking into account
$$
M=\int_{0}^{G(\alpha,\sigma)-\alpha}\frac{{\rm d}t}{\sqrt{F(\alpha)-F(\alpha+t)}},~~t=\varphi-\alpha,~~\frac{\partial G}{\partial\alpha}<1,
$$
 and $f(\varphi)<0$ is decreasing in $\varphi\in(\alpha,\infty)$,  we have
 \begin{align*}
\frac{{\rm d} M}{{\rm d}\alpha}=&\frac{1}{\sqrt{F(\alpha)-F(G)}}\left(\frac{\partial G}{\partial\alpha}-1\right)+\frac{1}{2}
\int_0^{G(\alpha,\sigma)-\alpha}{\frac{f(\alpha+t)-f(\alpha)}{(F(\alpha)-F(\alpha+t))^{3/2}}}{\rm d}t<0,
\end{align*}
that is, $M$ is decreasing in $\alpha$. Thanks to the monotonicity of $f(\varphi)$ and (\ref{B}), one gets
\begin{align}\label{Mgu}
-\frac{\sqrt{2}\sigma}{f(G)}<\frac{2\sqrt{G-\alpha}}{\sqrt{-f(G)}}<M(\alpha)<\frac{2\sqrt{G-\alpha}}{\sqrt{-f(\alpha)}}<-\frac{\sqrt{2}\sigma}{f(\alpha)}.
\end{align}
Hence, in each case of (I), (II) and (III), we see that as $\alpha\to \infty$, $M$ will approaches, respectively,
$$
0,\; -\frac{\sqrt{2}\sigma}{q_0},\; 0.
$$
Similarly, we also have $\lim_{\alpha\to \varphi_e}{M}=\infty$ in cases (I) and (II) and $\lim_{\alpha\to -\infty}M=-{\sqrt{2}\sigma}/{q_0}$ in case (III). See Figure \ref{Ma} for a sketch of the function $M(\alpha)$ in each case. The existence and uniqueness of $\alpha=\alpha(\sigma,L)$ then follow. This completes the proof.
\end{proof}

Parts of our aims in this work are interested in how $\sigma$ and $L$ affect several physical quantities. To this end, we need more information on the relation between $\sigma,~L$ and $\alpha$ defined by (\ref{re2}), which plays a
crucial role in Section 3.

\begin{prop}\label{tha}
The function  $\alpha=\alpha(\sigma,L)\in\mathbb{R}$ defined by (\ref{re2})  is increasing in $\sigma$ and is decreasing in $L$. Furthermore, in each case of Theorem \ref{thN}, one has
\begin{align*}
&
\lim_{L\to 0}\alpha=\infty,\;\lim_{\sigma\to \infty}\alpha=\alpha^L,\;
\lim_{L\to \infty}\alpha=\lim_{\sigma\to 0}\alpha=\varphi_e\;\mbox{ for case (I)},\\
&
\lim_{L\to -\frac{2\sigma}{q_0}}\alpha=\lim_{\sigma\to -\frac{q_0L}{2}}\alpha=\infty,\quad\lim_{L\to \infty}\alpha=\lim_{\sigma\to 0}\alpha=\varphi^*\; \mbox{ for case (II)},\\
&
\lim_{L\to 0}\alpha=\infty,\; \lim_{\sigma\to \infty}\alpha=\alpha^L,\;
\lim_{L\to -\frac{2\sigma}{q_0}}\alpha=\lim_{\sigma\to -\frac{q_0L}{2}}\alpha=-\infty\; \mbox{ for case (III)},
\end{align*}
where $\alpha^L$ is  uniquely determined by

 \begin{align}\label{alpha}
\frac{L}{\sqrt{2}}=\int_{\alpha^L}^{\infty}\frac{{\rm d}\varphi}{\sqrt{F(\alpha^L)-F(\varphi)}}.
\end{align}
\end{prop}

\begin{proof}
The monotonicity of $\alpha=\alpha(\sigma,L)$ with respect to $L$ directly follows from the graph of $M(\alpha)$.
Suppose $\alpha$ is not increasing in $\sigma$. Then there exist $\sigma_1>\sigma_2,~\alpha_1\leq\alpha_2$ satisfying
$$
M(\alpha_1,\sigma_1)=M(\alpha_2,\sigma_2)=\frac{L}{\sqrt{2}}.
$$
By $\frac{\partial M}{\partial \alpha}<0$, we get $M(\alpha_1,\sigma_1)\geq M(\alpha_2,\sigma_1)$ which leads to $M(\alpha_2,\sigma_2)\geq M(\alpha_2,\sigma_1)$, that is,
$$
  \int_{\alpha_2}^{G(\alpha_2,\sigma_2)}\frac{{\rm d}\varphi}{\sqrt{F(\alpha_2)-F(\varphi)}}\geq\int_{\alpha_2}^{G(\alpha_2,\sigma_1)}\frac{{\rm d}\varphi}{\sqrt{F(\alpha_2)-F(\varphi)}}.
$$
However, by $\frac{\partial G}{\partial \sigma}>0$, we have $G(\alpha_2,\sigma_2)<G(\alpha_2,\sigma_1)$ which yields a contradiction. Hence, $\alpha$ is increasing in $\sigma$.

We first prove that if $\min_{1\leq j\leq N}\{z_j\}<0$ (corresponding to cases (I) and (III)), then $\lim_{\sigma\to \infty}\alpha$ is convergent to a finite number $\alpha^L$. Otherwise one has $\lim_{\sigma\to \infty}\alpha=\infty$ due to the monotonicity of $\alpha$ with respect to $\sigma$. On the one hand, we have
 \begin{align*}
\frac{L}{\sqrt{2}}<I(\alpha):=\int_{\alpha}^{\infty}\frac{{\rm d}\varphi}{\sqrt{F(\alpha)-F(\varphi)}}.
\end{align*}
Suppose, for some $1\le l\le N$, $z_l=\min_{1\leq j\leq N}\{z_j\}<0$. The integral $I(\alpha)$ is convergent since
\[F(\alpha)-F(\varphi)\sim c_l^be^{-z_l\varphi},~~\textit{as}~~\varphi\to \infty.\]
We only need to show $\lim_{\alpha\to \infty}I(\alpha)=0$, which contradicts with $\lim_{\sigma\to \infty}\alpha=\infty$.
Indeed, by
$$
I_1(\alpha):=\int_{\alpha}^{\alpha+1}\frac{{\rm d}\varphi}{\sqrt{F(\alpha)-F(\varphi)}}<\int_{\alpha}^{\alpha+1}\frac{{\rm d}\varphi}{\sqrt{f(\alpha)(\alpha-\varphi)}}=\frac{2}{\sqrt{-f(\alpha)}},
$$
one has $\lim_{\alpha\to \infty}I_1(\alpha)=0$.
Denote by $\theta=1-e^{z_l/2}\in(0,1)$. Combing with the fact that $F(\varphi)$ is decreasing in $\varphi\in(\alpha,\infty)$ and $\lim_{\alpha\to \infty}F(\alpha+1)/F(\alpha)=e^{-z_l}$, we see that if $\varphi>\alpha+1$ and $\alpha$ is sufficiently large then $F(\varphi)$ is negative and
\begin{align*}
F(\alpha)-F(\varphi)&=-\theta F(\varphi)+\left((1-\theta)\frac{F(\varphi)}{F(\alpha)}-1\right)(-F(\alpha)),\\
                 &\geq -\theta F(\varphi)+\left((1-\theta)\frac{F(\alpha+1)}{F(\alpha)}-1\right)(-F(\alpha)),\\
                 &\geq-\theta F(\varphi),
\end{align*}
which yields
$$
I_2(\alpha):=\int_{\alpha+1}^{\infty}\frac{{\rm d}\varphi}{\sqrt{F(\alpha)-F(\varphi)}}<\int_{\alpha+1}^{\infty}\frac{{\rm d}\varphi}{\sqrt{-\theta F(\varphi)}},
$$
and $\lim_{\alpha\to \infty}I_2(\alpha)=0$.  Therefore we obtain
$$
\lim_{\alpha\to \infty}I(\alpha)=\lim_{\alpha\to \infty}(I_1(\alpha)+I_2(\alpha))=0,
$$
 which is a contradiction. So one has  $\lim_{\sigma\to \infty}\alpha=\alpha^L$ for some constant $\alpha^L$.

The remaining limits can be derived easily from properties of $M(\alpha)$ and (\ref{Mgu}).
\end{proof}

\begin{rem}
Our approach can be employed to investigate the problem of two parallel plates with opposite charges, namely N-BVP (\ref{PB}) and (\ref{bcc2}) satisfying $\sigma_1\sigma_2>0$, although it is not the main focus of this paper. Without loss of generality, providing that $0<\sigma_1<\sigma_2$, one can use similar arguments and prove that this N-BVP has a unique solution if and only if one of the following conditions holds:

\begin{itemize}
\item[{\rm(I)}] either $z_iz_j<0$ for some $i,j\in\{1,\cdots,N\}$ or $z_i<0$ for all $i\in\{1,\cdots,N\}$ and $q_0\geq0$;


\item[{\rm(II)}] $z_i>0$ for all $i\in\{1,\cdots,N\}$, $q_0<0$ and
$
 L>-\frac{\sigma_2-\sigma_1}{q_0};
$

\item[{\rm(III)}] $z_i<0$ for all $i\in\{1,\cdots,N\}$, $q_0<0$ and
$
L<-\frac{\sigma_2-\sigma_1}{q_0}.
$
\end{itemize}

\end{rem}

\section{Critical features and their mechanisms}\label{CF}
\setcounter{equation}{0}
Major findings in this work will be discussed in this section, which are new critical features of some important physical
quantities and the significance of  whether or not an equilibrium is presented in the PB equation.
The findings rely on a detailed analysis of the BVPs with attention on physical implications. More precisely, we will focus on mechanism of GC layer in \S\ref{layer}, saturation of surface charge density in $L$ in \S\ref{secSCD},  critical property of electric pressure in  \S\ref{sEP},  and critical length  for the switching of electric force direction in \S\ref{critL}. Some of these features are closely related.


\subsection{Mechanism of  GC layer: Presence of equilibrium.}\label{layer}

The Gouy-Chapman layer is an important phenomenon of  electrostatic forces in charged membrane systems in the electrochemistry theory. It quantified the striking double-layer feature near the surface-charge of electrolytes. This was first observed by Gouy and Chapman independently based on explicit solutions of the PB equation in a special setting (\cite{Cha13, Gouy10}). The Gouy-Chapman theory has been further modified and tested and applied in real systems. Nevertheless, mechanisms of
 Gouy-Chapman layer were not well-studied or well-exposed  to the best knowledge of the authors.

The Gouy-Chapman layer initially stems from the investigation of the one plate problem, that is, the plate located at $x=0$ immersed in an infinitely long electrolyte
with left boundary being either surface charge (Section \ref{GCN}) or surface potential (Section
\ref{GCD}) and right boundary being zero electric field intensity.
Then one has the PB equation with the  boundary condition
\begin{align}
&\mbox{Neumann: }\quad  \frac{{\rm d}\varphi}{{\rm d}x}(0)=-\sigma,~~\frac{{\rm d}\varphi}{{\rm d}x}(\infty)=0.\label{bv1}
 \end{align}

 The dependence of the surface potential $\varphi(0)=\varphi_s$ on the  surface charge density $\sigma$ was the original interest of, independently, Gouy and Chapman on this problem and, the dependence is often called Grahame's equation \cite{Is1992}. However, the arguably most important discovery by Gouy and Chapman is the so-called {\em Gouy-Chapman (GC) layer:} There is a finite number $\delta_{GC}>0$, {\em the Gouy-Chapman length,} such that one-half of the total charge in electrolyte are accumulated in the narrow
region next to the left boundary within $\delta_{GC}$ distance.

 What is the mechanism behind the GC layer? To the best knowledge of the authors, there is no investigation on this question. By chance, in our study of PB theory, we found the hidden mechanism of the GC layer is the presence of equilibrium (finite or infinite) in PB equation, although the equilibrium is located far away from the GC layer.

 To present this mechanism of the GC layer, we will treat the problem of PB with boundary conditions (\ref{bv1})
 as $L\to\infty$ of PB with Neumann conditions
  \begin{align}
  \frac{{\rm d}\varphi}{{\rm d}x}(0)=-\sigma\;\mbox{ and }\; \frac{{\rm d}\varphi}{{\rm d}x}\big(\frac{L}{2}\big)=0.\label{bv11S}
 \end{align}
 By relaxing the original setting of Gouy and Chapman to the setting in this paper, in particular allowing a fixed charge,  we find that {\em there is a GC layer if and only if the PB system has an equilibrium, a finite equilibrium or an equilibrium-at-infinity.} In fact, the finding follows easily from a simple geometric observation of the phase plane portrait.

 Note that, by the translation $x\to x-L/2$, PB with (\ref{bv11S}) is converted to PB with
  \begin{align}
  \frac{{\rm d}\varphi}{{\rm d}x}(-\frac{L}{2})=-\sigma\;\mbox{ and }\; \frac{{\rm d}\varphi}{{\rm d}x}(0)=0.\label{bv11}
 \end{align}

\subsubsection{GC layer with  symmetric Neumann boundary conditions.}\label{GCN}

Consider PB with  symmetric Neumann boundary conditions (\ref{nbc0}) (equivalent to (\ref{bv11})). The total (net) charge between two plates is
\begin{align*}
Q_{total}=&\int_{-\frac{L}{2}}^{\frac{L}{2}}\Big(q_0+\sum_{j=1}^Nz_jc_j(x)\Big){\rm d}x=\int_{-\frac{L}{2}}^{\frac{L}{2}}\Big(q_0+\sum_{j=1}^Nz_jc_j^be^{-z_j\varphi(x)}\Big){\rm d}x\\
=&-\int_{-\frac{L}{2}}^{\frac{L}{2}} \varphi''(x){\rm d}x
=\varphi'\Big(-\frac{L}{2}\Big)-\varphi'\Big(\frac{L}{2}\Big)=-2\sigma.
\end{align*}
Let $\delta_{N}$ be the width of the ``layer'' next to the left (or right) plate over which $1/4$ of the ions are accumulated. (The subscript $N$ in $\delta_N$ indicates that the quantity is associated with Neumann boundary conditions and we will later use $\delta_D$ for that associated with Dirichlet boundary conditions.)

Note that the choice of $1/4$ results in the accumulation of charges at the left and right layers amounting to $1/2$ of the total. Otherwise, the choice is entirely arbitrary. In other words, any number $\rho\in (0, 1/2)$ can be opted for instead of $1/4$ to achieve the identical characteristic of multiple length scales/layers.

The total  charge over the union
\[\left(-\frac{L}{2}, -\frac{L}{2}+\delta_{N}\right)\cup \left(\frac{L}{2}-\delta_{N},\frac{L}{2}\right)\]
 is thus one-half of the total charge, and then
\[\frac{1}{4}Q_{total}=\int_{-L/2}^{-L/2+\delta_N}\Big(q_0+\sum_{j=1}^Nz_jc_j^be^{-z_j\varphi(x)}\Big){\rm d}x=\int_{L/2-\delta_N}^{L/2}\Big(q_0+\sum_{j=1}^Nz_jc_j^be^{-z_j\varphi(x)}\Big){\rm d}x,\]
or equivalently,
\[\varphi'(-\frac{L}{2}+\delta_N)=-\varphi'(\frac{L}{2}-\delta_N)=-\frac{\sigma}{2}.\]
Taking integrations for $x\in (-L/2,0)$ and $x\in (-L/2,-L/2+\delta_N)$ from
\[\sqrt{2}{\rm d}x=\frac{{\rm d}\varphi}{\sqrt{F(\alpha)-F(\varphi)}}\]
and making use of the time-map (\ref{re2}) 
to get,
\begin{align}
\frac{L}{\sqrt{2}}=\int_{\alpha}^{G(\alpha,\sigma)}\frac{{\rm d}\varphi}{\sqrt{F(\alpha)-F(\varphi)}},\label{gc10}\\
{\sqrt{2}}{\delta_N}=\int_{G(\alpha,\sigma/2)}^{G(\alpha,\sigma)}\frac{{\rm d}\varphi}{\sqrt{F(\alpha)-F(\varphi)}},\label{gc20}
\end{align}
where the function $G(\cdot,\cdot)$ is defined in Lemma \ref{funB}.

\begin{thm}\label{thgc}
Fix $\sigma>0$. The function $\delta_N=\delta_N(L)$, defined by (\ref{gc20}), satisfies
\begin{itemize}
\item[{\rm(A)}]  in the presence of  equilibrium,
\begin{align}\label{deltaN}
\lim_{L\to \infty}\delta_N(L)=\delta_N^{\infty}:=\frac{1}{\sqrt{2}}\int_{F^{-1}(F_e-\frac{\sigma^2}{8})}^{F^{-1}(F_e-\frac{\sigma^2}{2})}\frac{{\rm d}\varphi}{\sqrt{F_e-F(\varphi)}}<\infty,
\end{align}
where $F^{-1}(\cdot)$ is the inverse  of $F(\cdot)$ on the monotonic decreasing interval;

\item[{\rm(B)}] in the absence of equilibrium,
\begin{align}\label{deltaN2}
\lim_{L\to L_c}\delta_N(L)=\frac{L_c}{4}~~\mbox{where}~~L_c=-\frac{2\sigma}{q_0}>0.
\end{align}
\end{itemize}
\end{thm}

\begin{proof}
{\rm (A)}
By utilizing Proposition \ref{tha} and letting $L$ tend to infinity, it follows that $\alpha$ approaches $\varphi_e$. Consequently, we observe that $\delta_N$ tends to $\delta_{N}^{\infty}$ as $L$ tends to infinity.

{\rm (B)}  By Proposition \ref{tha}, we know that $L$ has  an upper bound $L_c$ and  cannot tend to infinity   in the non-equilibrium case.
Furthermore, taking (\ref{gc10}) and Theorem \ref{tha} into account, one has $\alpha\to-\infty$ as $L\to L_c$.
Clearly, one also has
\begin{align*}
\frac{\sqrt{-f(G(\alpha,\sigma/2))}}{\sqrt{2}}<\frac{\sqrt{G(\alpha,\sigma)-\alpha}-\sqrt{G(\alpha,\sigma/2)-\alpha}}{\delta_N}<\frac{\sqrt{-f(G(\alpha,\sigma))}}{\sqrt{2}}.
\end{align*}
It follows from (\ref{B}) that
$$
\lim_{\alpha\to-\infty}\big(G(\alpha,\sigma)-\alpha\big)=-\frac{\sigma^2}{2q_0}\;\mbox{ and }\;\lim_{\alpha\to-\infty}\big(G(\alpha,\sigma/2)-\alpha\big)=-\frac{\sigma^2}{8q_0}.
$$
The proof can then be easily completed.
\end{proof}
\begin{rem}\label{ree}
The  quantity $\delta_N^{\infty}$ depends on the surface charge $\sigma$.  As indicated by (\ref{deltaN}), it is evident that $\delta_N^{\infty}$ approaches zero as  $\sigma$
goes to infinity.

\end{rem}
Theorem \ref{thgc} states that when the amount of charges between two plates remains the same,
as the separation distance $L$ increases, the  length scale of $\delta_N=\delta_N(L)$  depends critically on
whether or not PB equation has an equilibrium. More precisely, if PB admits an equilibrium (saddle or equilibrium-at-infinity), then
\[\frac{\delta_N(L)}{L}\to 0\;\mbox{ as }\; L\to\infty;\]
that is,   $\delta_N(L)$ and $L$ have {\em different} (length) scales. That is the reason for the terminology of ``layer''.
 On the other hand,  if there is no equilibrium, then the BVP has a classical solution if and only if $L\leq  L_c$, and
 \[
 \frac{\delta_N(L)}{L}\to \frac{1}{4}\;\mbox{ as } L\to L_c.
 \]
For fixed $\sigma$, one cannot take $L\to\infty$ in the absence of equilibrium case. In the above restricted sense, we conclude that \emph{there is no GC layer if the PB equation has no equilibrium}. This mechanism for GC layer is more direct for Dirichlet boundary conditions to be discussed next, where one can take $L\to\infty$ also in the absence of equilibrium case.

\subsubsection{GC layer with  symmetric Dirichlet boundary conditions.}\label{GCD}
If one keeps the value of surface potential, that is,
 \begin{align}\label{bcs}
\varphi\left(-\frac{L}{2}\right)=\varphi\left(\frac{L}{2}\right)=\varphi_s,
\end{align}
and increases the separation length $L$, one may take another look at the crucial effect of equilibrium on the GC layer.
Similar to the Neumann boundary condition case discussed in Section \ref{GCN}, let $\delta_{D}$ be the width of the ``layer'' next to the left (or right) plate over which $1/4$ of the ions are accumulated; that is,
$$
\frac{1}{4}Q_{total}=\int_{-L/2}^{-L/2+\delta_D}(q_0+\sum_{i=1}^Nz_ic_i^be^{-z_i\varphi(x)}){\rm d}x=\int_{L/2-\delta_D}^{L/2}(q_0+\sum_{i=1}^Nz_ic_i^be^{-z_i\varphi(x)}){\rm d}x.
$$
Using the time-map (\ref{re2}) again, one has
\begin{align}
\frac{L}{\sqrt{2}}&=\int_{\alpha}^{\varphi_s}\frac{{\rm d}\varphi}{\sqrt{F(\alpha)-F(\varphi)}},\quad
{\sqrt{2}}{\delta_D}=\int_{F^{-1}(\frac{F(\varphi_s)+3F(\alpha)}{4})}^{\varphi_s}\frac{{\rm d}\varphi}{\sqrt{F(\alpha)-F(\varphi)}}.\label{gc22}
\end{align}
\begin{thm}\label{thgc2}
Fix  $\varphi_s$. The function $\delta_D=\delta_D(L)$, defined by (\ref{gc22}), satisfies

\begin{itemize}
\item[{\rm(A)}]  in the presence of  equilibrium, one has
\begin{align}\label{deltaD}
\lim_{L\to \infty}\delta_D(L)=\delta_{D}^{\infty}:=\frac{1}{\sqrt{2}}\int_{{  F^{-1}(\frac{F(\varphi_s)+3F_e}{4})  }}^{\varphi_s}\frac{{\rm d}\varphi}{\sqrt{F_e-F(\varphi)}}<\infty,
\end{align}
where $F^{-1}(\cdot)$ is the inverse  of $F(\cdot)$ on the monotonic decreasing interval;
in particular,
$\lim_{L\to \infty}\frac{\delta_D(L)}{L} =0$; that is,
$\delta_D(L)$ and $L$ have different length scales, and there is GC-layer in the presence of  equilibrium;

\item[{\rm(B)}]  in the absence of equilibrium, one has $\lim_{L\to \infty}\delta_D(L)=\infty$; in fact,
\[
 \frac{1}{4}\leq\liminf_{L\to \infty}\frac{\delta_D(L)}{L}\leq\limsup_{L\to \infty}\frac{ \delta_D(L)}{L}\leq1,\]
in particular, $\delta_D(L)$ and $L$ have the same length scale and there is no layer in the absence of equilibrium.
\end{itemize}
\end{thm}
\begin{proof}
(A) Considering (\ref{thequ}) and following the arguments employed in the proof of Theorem \ref{thgc}, one can readily establish its veracity.

(B) Set $\hat \varphi={  F^{-1}(\frac{F(\varphi_s)+3F(\alpha)}{4})}$. Then we have
\[\lim_{\alpha\to-\infty}\hat \varphi=-\infty,~~\frac{{\rm d}\hat\varphi}{{\rm d}\alpha}=\frac{3f(\alpha)}{4f(\hat\varphi)},\]
which implies
\[\lim_{\alpha\to-\infty}\sqrt{\frac{\hat\varphi-\alpha}{\varphi_s-\alpha}}=\frac{3}{4}.\]
The  statement (B) of Theorem \ref{thgc2} follows immediately from
\begin{align*}
\delta_D\geq \frac{\sqrt{2(\varphi_s-\alpha)}}{\sqrt{-f(\varphi_s)}}\left(1-\sqrt{\frac{\hat\varphi-\alpha}{\varphi_s-\alpha}}    \right),~~
L\leq \frac{\sqrt{2(\varphi_s-\alpha)}}{\sqrt{-f(\alpha)}}.
\end{align*}
This completes the proof.
\end{proof}

 Once again, for the Dirichlet BVP, there is a GC-layer  if and only if an equilibrium is present for the PB system.
Furthermore, we will show that presence of equilibrium is also responsible for a number of other critical features.

\subsection{Surface charge density for surface potentials}\label{secSCD}
 In this part,  we examine properties of surface charge density for  surface potentials.
For simplicity, we restrict our analysis to the case where the surface potentials are equal. However, it should be noted that similar results hold for the case where the surface potentials are not equal (see Theorems \ref{thnon1}-\ref{thnon3} in Section \ref{critL}).

\subsubsection{Saturation of surface charge density }\label{Saturation}
Providing $\varphi_0=\varphi_1$,  one has $-\varphi'(-L/2)=\varphi'(L/2)$ and, by Coulomb's law, it equals the surface charge density at $x=-L/2$ (or the negative surface charge density at $x=L/2$).  We are interested in the dependence of the magnitude
\begin{align}\label{SCD}
\sigma(L,q_0):=|\varphi'(-L/2)|=|\varphi'(L/2)|
\end{align}
 of the surface charge density on the length $L$ and on large fixed charge $q_0$. We also simply refer to $\sigma(L,q_0)$ as the surface charge density.

Intuitively, to maintain the surface potential $\varphi_0=\varphi_1$, the surface charge $\sigma(L,q_0)$ should increase as the length $L$ does.
On the other hand, the limiting behavior of $\sigma(L,q_0)$ as $L\to \infty$ depends critically on wether or not an equilibrium is present. More precisely, we have

\begin{thm}\label{thequ1}
Assume  $\varphi_0=\varphi_1$ and let $\varphi(x)$ be the unique solution of D-BVP (\ref{PB})-(\ref{bc}). Then,
\[\sigma(L,q_0)=\max\{|\varphi'(x)|: x\in [-L/2,L/2]\}\]
 and $\sigma(L,q_0)$ is increasing in $L$.
Furthermore, one has
\begin{itemize}
\item[{\rm(A)}]  in the presence of  equilibrium,
$\displaystyle{
\lim_{L\rightarrow\infty}\sigma(L,q_0)=\sqrt{2(F_e-F(\varphi_0))}<\infty;}
$

\item[{\rm(B)}]  in the absence of  equilibrium,
$\displaystyle{\lim_{L\rightarrow\infty}\sigma(L,q_0)=\infty.}$
\end{itemize}
\end{thm}


\begin{proof}
We consider  the saddle case (i) of Proposition \ref{th1} and $\varphi_0<\varphi^*$.
From the phase plane portraits, one can easily see that  $|\varphi'(x)|$ takes the maximum at $x=\pm L/2$,
which leads to
\[
\sigma(L,q_0)=\max\{|\varphi'(x)|: x\in [-L/2,L/2]\}.
\]
It follows from
the proof of Theorem \ref{thequ} that the middle potential $\alpha=\varphi(0)$ increases as $L$ increases, in particular,
$\alpha$ approaches to the equilibrium $\varphi^*$ as $L$ goes to infinity. Hence one has
$$
\lim_{L\rightarrow\infty}\sigma(L,q_0)=\lim_{L\rightarrow\infty}\sqrt{2(F(\alpha)-F(\varphi_0))}=\sqrt{2(F_e-F(\varphi_0))}.
$$

Proceeding as above, one can  easily prove the remaining cases.
\end{proof}

An interesting physical interpretation of Theorem \ref{thequ1}  is that, for given potential $\varphi_0=\varphi_1$ at the two surfaces, the magnitude  $\sigma(L,q_0)$ of the surface charge density increases in $L$ as expected. But, for the cases with either a saddle equilibrium or  the equilibrium-at-infinity, the existence of the (finite) limit
\[\sigma(q_0):=\lim_{L\rightarrow\infty}\sigma(L,q_0)=\sqrt{2(F_e-F(\varphi_0))}<\infty\]
may not be well expected. It asserts that, in the presence of equilibrium,
 the surface charge density {\em  saturates}  with the \emph{maximum surface charge density} $\sigma(q_0)$ for large $L$.
  On the other hand,  when there is No equilibrium, there is No saturation of surface charge density $\sigma(L,q_0)$ for large $L$.

 The saturation phenomenon is rather a surprise (at least to the authors).  However, it can be understood easily and intuitively
 from the phase plane portraits (a), (b) and (c) in Figure \ref{figp}.
Indeed, for the case with a saddle equilibrium $(\varphi^*,0)$,
the stable and unstable manifolds (solid curves in Figure \ref{figp} (a)) of the equilibrium $(\varphi^*,0)$ form a sector that can hold the orbit in between, in the sense that, as $L$ increases, the middle part of the orbit gets closer to and takes longer time (length) near  the equilibrium $(\varphi^*,0)$. As a consequence, for given surface potentials, the surface charge stays bounded uniformly in $L$.
For the case with the  equilibrium-at-infinity, although the middle part of the orbit moves toward infinity as $L$ increases, the orbit is always bounded by the stable and unstable manifolds of the equilibrium-at-infinity.

Whereas when no equilibrium is presented, say,  the case in Figure \ref{figp} (d), for given potential $\varphi_0=\varphi_1$ at the two surfaces, as $L$ increases, the maximal value of $\varphi$ moves to infinity, and hence, the value $\varphi'(-L/2)$ goes to infinity too.
This is due to the fact that there is no equilibrium for relaxation and no stable/unstable manifolds to bound the orbit as $L$ increases.

The above observations from the phase plane portraits reveal that the mechanism for saturation of surface charge density in large $L$ is the presence of the equilibrium (both finite and infinite).

%

\subsubsection{Maximum surface charge density on large $q_0$}
We now examine, in the saddle case, how   the maximum surface charge density
 \begin{align}\label{SConq}
 \sigma(q_0):=\lim_{L\rightarrow\infty}\sigma(L,q_0)=\sqrt{2(F_e-F(\varphi_0))}
 \end{align}
  depends on strongly large fixed charges $q_0$ through $\varphi^*$. (The equilibrium-at-infinity case has $q_0=0$.) It turns out $\sigma(q_0)\to \infty$ as $|q_0|\to \infty$. Our result (Proposition \ref{rate4MC} below) provides the ``rates'' at which $\sigma(q_0)$ approaches infinity.

  For convenience, we use the notation $\varphi^*=\varphi^*(q_0)$
to indicate the dependence of $\varphi^*$ on $q_0$ and recall that $\varphi^*(q_0) $ is uniquely determined by
\begin{align}\label{q2phistar}
q_0+ \sum_{j=1}^{N} z_jc_{i}^{b}e^{-z_j\varphi^*(q_0)}=0.
\end{align}

\begin{prop}\label{rate4MC}

Assume the PB system (\ref{Hpb}) admits  a saddle equilibrium.

\begin{itemize}
\item[\rm{(N)}] If $z_{min}:=\min_{1\leq j\leq N}\{z_j\}<0$, then
\[\lim_{q_0\to \infty}  \frac{\sigma^2(q_0)}{q_0\ln q_0}=-\frac{2}{z_{min}}.\]

\item[\rm{(P)}] If $z_{max}=\max_{1\leq j\leq N}\{z_j\}>0$, then
    \[\lim_{q_0\to -\infty}  \frac{\sigma^2(q_0)}{q_0\ln |q_0|}=-\frac{2}{z_{max}}.\]
\end{itemize}
\end{prop}

\begin{proof}
(N) It follows from (\ref{q2phistar}) that $\varphi^*=\varphi^*(q_0)$ satisfies
\begin{align}\label{q02}
\frac{{\rm d}\varphi^*}{{\rm d}q_0}=\frac{1}{\sum_{i=1}^{N} z_i^2c_{i}^{b}e^{-z_i\varphi^*(q_0)}}>0,~~\lim_{q_0\to  \infty}\varphi^*= \infty,~~\lim_{q_0\to \infty}\frac{{\rm d}\varphi^*}{{\rm d}q_0}= 0.
\end{align}

Then, by the  L'Hopital's rule, one obtains
\begin{align*}
\lim_{q_0\to \infty}&  \frac{F(\varphi^*(q_0))-F(\varphi_0)}{q_0\ln q_0}=\lim_{q_0\to \infty} \frac{\varphi^*(q_0)}{\ln q_0}-\lim_{q_0\to \infty} \frac{ \sum_{i=1}^{N} c_{i}^{b}e^{-z_i\varphi^*(q_0)} }{q_0\ln q_0}\\
=&\lim_{q_0\to \infty} q_0\frac{{\rm d}\varphi^*}{{\rm d}q_0}+\lim_{q_0\to + \infty} \frac{\sum_{i=1}^{N} c_{i}^{b}e^{-z_i\varphi^*(q_0)}}{\ln q_0(\sum_{i=1}^{N} c_{i}^{b}z_ie^{-z_i\varphi^*(q_0)})}\\
=&-\lim_{\varphi^*\to \infty} \frac{\sum_{i=1}^{N} c_{i}^{b}z_ie^{-(z_i-z_{min})\varphi^*}}{\sum_{i=1}^{N} c_{i}^{b}z_i^2e^{-(z_i-z_{min})\varphi^*}}+\lim_{\varphi^*\to  \infty} \frac{\sum_{i=1}^{N} c_{i}^{b}e^{-(z_i-z_{min})\varphi^*}}{\ln q_0(\sum_{i=1}^{N} c_{i}^{b}z_ie^{-(z_i-z_{min})\varphi^*})}\\
=&-\frac{1}{z_{min}}.
\end{align*}

(P) It can be established in a similar way.
\end{proof}

 \subsection{Electric pressure for Neumann boundary conditions}\label{sEP}

  {\em Electric pressure} $P$ of an electrolytes in a region is the functional derivative of the total electrostatic energy $F$ with respect to the volume. Electric pressure is 
one of the {\em measurable} physical quantities of an electrolytes -- {\em an important feature of electric pressure}. For two-plate problem considered in this work,  the functional derivative is taken with respect to 
 the  distance between the two plates, and 
 the electric pressure measures the force from electrostatic interaction acting on the left plate A (or equivalently
the right plate B)  per unit area as indicated  in Figure \ref{figf}.

 We now present a quick derivation of the electric pressure $P$ for the setup of two symmetric plates, following the treatment in \cite{ep3,Marcus1955,xing}.
There are several ways to obtain the formula of electric pressure. One of them is the derivation of the variation of
the  electrostatic free energy with respect to the inter-membrane distance.
Consider the electrostatic free energy
\begin{align}
\mathcal{F}=&\underbrace{\int_{-\frac{L}{2}}^{\frac{L}{2}}\left(-\frac{\varepsilon_r(x)\varepsilon_0}{2}\Big(\frac{{\rm d}\varphi}{{\rm d}x}\Big)^2+e_0\varphi(x)\Big(q(x)+\sum_{j=1}^Nz_jc_j(x) \Big)-\sum_{j=1}^N\mu_j^bc_j(x)   \right){\rm d}x}_{\mathcal{U}}\label{energy}
\\
+&\underbrace{\int_{-\frac{L}{2}}^{\frac{L}{2}}\left(k_BT\sum_{j=1}^Nc_j(x)(\ln c_j(x)-1)\right)}_{\mathcal{-TS}}{\rm d}x,\notag
\end{align}
where $\mu_j^b=k_BT\ln c_j^b$ is the bulk electrochemical potential, $\mathcal{U}$ is the internal energy and $\mathcal{-TS}$ is the entropy. Here the function $\varphi(x)$ satisfies the Poisson equation (\ref{gPBe})
and $c_j(x)$ satisfies the Boltzmann distributions (\ref{dist}), which corresponds to the
equilibrium of the free energy functional $\mathcal{F}$. Indeed, the  variation $\delta \mathcal{F}/\delta {\varphi}=0$ leads to  (\ref{gPBe}), and the variation $\delta \mathcal{F}/\delta {c_j}=0$ leads to
(\ref{dist}).  Providing that $\varepsilon_r(x)=\varepsilon_r$ and $q(x)=q_0$ being constants, we rescale the coordinate $x\to X=x/L+1/2$  and convert (\ref{energy}) into
\begin{align}\label{energy2}
\mathcal{F}=&\int_{0}^{1}\left(-\frac{\varepsilon_r\varepsilon_0}{2L}\left(\frac{{\rm d}\varphi}{{\rm d}X}\right)^2+e_0\varphi L\left(q_0+\sum_{j=1}^Nz_jc_j\right)-L\sum_{j=1}^N\mu_j^bc_j    \right){\rm d}X
\\
+&\int_{0}^{1}\left(k_BTL\sum_{j=1}^Nc_j(\ln c_j-1)\right){\rm d}X.
\end{align}

Then the  electric pressure is given by
\begin{align}
P=-\frac{\partial \mathcal{F}}{\partial L}=&-\int_{0}^{1}\left(\frac{\varepsilon_r\varepsilon_0}{2L^2}\left(\frac{{\rm d}\varphi}{{\rm d}X}\right)^2+e_0\varphi \Big(q_0+\sum_{j=1}^Nz_jc_j\Big)-\sum_{j=1}^N\mu_j^bc_j    \right){\rm d}X  \notag
\\
&-\int_{0}^{1}\left(k_BT\sum_{j=1}^Nc_j(\ln c_j-1)\right){\rm d}X,\label{EP}
\end{align}
Owing to the  Boltzmann distributions (\ref{dist}), we eliminate the variables $c_j$  from (\ref{EP}), obtaining
\begin{align}
P=&-\int_0^1  \left( \frac{\varepsilon_r\varepsilon_0}{2L^2}\left(\frac{{\rm d}\varphi}{{\rm d}X}\right)^2+e_0q_0\varphi-k_BT\sum_{j=1}^Nc_j^be^{-\frac{z_je_0}{k_BT}\varphi}\right){\rm d}X,\notag  
\\
=&-\frac{1}{L}\int_{-\frac{L}{2}}^{\frac{L}{2}}  \left( \frac{\varepsilon_r\varepsilon_0}{2}\left(\frac{{\rm d}\varphi}{{\rm d}x}\right)^2+e_0q_0\varphi-k_BT\sum_{j=1}^Nc_j^be^{-\frac{z_je_0}{k_BT}\varphi}\right){\rm d}x.\label{EP2}
\end{align}
The integrand in the definite integral on the right side of (\ref{EP2}) is, in fact, the first integral of the PB equation (\ref{PB}) before nondimensionalization. As a result, this integrand remains constant along the solution $\varphi(x)$ of the PB equation. 
Therefore, the electric pressure between the  charged plates is given by
  \begin{align}\label{pre}
P=-\frac{\varepsilon_r\varepsilon_0}{2}\left(\frac{{\rm d}}{{\rm d}x}\varphi(x_0)\right)^2-e_0q_0\varphi(x_0)+k_BT\sum_{j=1}^Nc_j^be^{-\frac{z_je_0}{k_BT}\varphi(x_0)},
\end{align}
where $x_0$ represents any point in the interval  $[-L/2,~L/2]$ and it is typically chosen to be $x_0=0$.


It should be pointed out that  $P$ is not the net electric pressure
between two plates since the ionic gas outside the plates also exerts pressure on
them \cite{xing}. The net electric pressure $P_{net}$ only differs a constant from $P$, providing that the PB
system has a (finite/infinity) equilibrium. Hence, we also call
$P$ the internal electric pressure.

Consider the PB equation with the Neumann boundary condition (\ref{nbc0}) and $\sigma>0$. An application of the first integral $\mathcal{H}$ and $\varphi'(0)=0$ leads to
\begin{lem}\label{epLem} Assume $\sigma>0$. In terms of  the dimensionless variables in (\ref{sca1})-(\ref{sca2}),
the electric pressure (after dropping the tildes) can be written as
\begin{align}\label{ep}
P=-k_BT\Lambda F(\alpha),
\end{align}
with $\Lambda=\sum_{i=1}^Nz_i^2c_i^b$ is the ionic strength of the mixture and $\alpha=\varphi(0)$.
\end{lem}

 Note that $F$ is decreasing in $\alpha$ for $\sigma>0$, see  (\ref{Hfun}) for the formula of $F$.  Therefore, $P$ is increasing in $\alpha$. The next result is a direct consequence of  Proposition \ref{tha} on $\alpha=\alpha(\sigma, L)$.

\begin{thm}\label{cor} Assume $\sigma>0$.
The function $P=P(\sigma,L)$  is increasing in $\sigma$ and is decreasing in $L$. Moreover,
corresponding to each case of Theorem \ref{thN}, one has
\begin{align*}
&\lim_{L\to 0}P=\infty,\;\lim_{\sigma\to \infty}P=P_L,\;\lim_{L\to \infty}P=\lim_{\sigma\to 0}P=P_e\;\mbox{ for case (I)},
\\
&
\lim_{L\to -\frac{2\sigma}{q_0}}P=
\lim_{\sigma\to -\frac{q_0L}{2}}P=\infty,\;\;\lim_{L\to \infty}P=\lim_{\sigma\to 0}P=P_e\; \mbox{ for case (II)},
\\
&\lim_{L\to 0}P=\infty,\; \lim_{\sigma\to \infty}P=P_L,\;
\lim_{L\to -\frac{2\sigma}{q_0}}P=
\lim_{\sigma\to -\frac{q_0L}{2}}P=-\infty\; \mbox{ for case (III)},
\end{align*}
where $P_e=-k_BT\Lambda F_e$ and $P_L=-k_BT\Lambda F(\alpha^L)$.
\end{thm}

From Theorem \ref{cor}, we have the following observations:

\begin{itemize}

\item[{\rm (i)}] Suppose that the value of $L$ is fixed. The  pressure $P$ increases as $\sigma$ increases.
In addition, if there are anions in electrolyte, that is, $z_j<0$ for certain $1\leq j\leq N$, the  pressure approaches
a finite value as $\sigma$ goes to infinity. By contrast, if there are only cations in electrolyte, that is, $z_j>0$ for all $1\leq j\leq N$, $\sigma$ has an  upper bound $-q_0L/2$, and the  pressure approaches infinity as $\sigma$  approaches the critical value $-q_0L/2$.

\item[{\rm (ii)}] Similarly, suppose that the value of $\sigma$ is fixed. The  pressure $P$ decreases as $L$ increases.
Furthermore, the  pressure approaches
a finite value as $L$ goes to infinity when the PB system has an equilibrium. However, if the PB system has no  equilibrium, $L$ has an  upper bound $-2\sigma/q_0$, and the  pressure approaches negative infinity as $L$  approaches the critical value $-2\sigma/q_0$.

\item[{\rm (iii)}] Fixing  the value of pressure $P$, or equivalently, fixing the value of $\alpha$,
one can view $L=L(\sigma)$ as a function in $\sigma$. Then, $L$ is increasing in $\sigma$.
Corresponding to {\rm(i)}, as $\sigma$  approaches to infinity, one sees that $L$ approaches
a finite value
\begin{align}\label{CritL2}
L_{max}=\sqrt{2}\int_{\alpha}^{\infty}\frac{1}{\sqrt{F(\alpha)-F(\varphi)}}{\rm d}\varphi;
\end{align}
if there are anions in electrolyte; whereas $L$ approaches infinity  if there are only cations in electrolyte.
Clearly, (\ref{CritL2}) coincides with (\ref{alpha}). Physically, this fact means that when the electrolyte contains anions, there exists a maximum separation between two plates
and the surface charge density yielding this
interaction at this maximal separation is infinity. For $L>L_{max}$, the value of fixed pressure $P$ cannot be fixed
and should become smaller. In addition, we also see that $L$ approaches zero when $\sigma$ approaches zero.

\end{itemize}

\subsection{Critical length for monotonicity of electric potential}\label{critL}
 Recall that the electric force at $x$ is the derivative $\varphi'(x)$ of the electric potential. Thus the monotonicity of the potential determines the direction of the electric force, and hence, is of significant. In this part, we will deal with the D-BVP (\ref{PB})-(\ref{bc}) when $\varphi_0\neq\varphi_1$ and determine the monotonicity of $\varphi(x)$. It turns out, once other parameters are fixed, there is a critical value of separation length for transitions between monotonic potential to non-monotonic potential. This is the concern of the present section.

 Note that if $\varphi(x)$ is a solution of the
D-BVP (\ref{PB})-(\ref{bc}), then $\widehat{\varphi}(x):=\varphi(-x)$ is a solution the D-BVP (\ref{PB}) with the boundary
value condition $\widehat\varphi(-L/2)=\varphi_1,~\widehat\varphi(L/2)=\varphi_0$. Thus without loss of generality we can only consider the case
$\varphi_0<\varphi_1$.  It can be seen from the phase portraits of system (\ref{Hpb}) that, unlike the equi-potential case, D-BVP with non-equal surface potential $\varphi_0< \varphi_1$ has two types of solutions, one is monotonically increasing in $x$ and another one is non-monotonic in $x$, see Figure \ref{nonequ}.

\begin{figure}[htpp]
\centering
\includegraphics[width=0.5\textwidth]{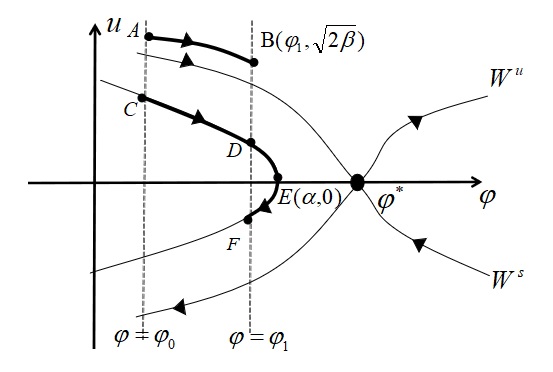}
\caption{
Two types of solutions in saddle case: monotone  solutions
represented by the arc AB, non-monotone solutions represented by the arc CEF.
}
\label{nonequ}
\end{figure}


Define  the following critical distance
$L_0=L_0(q_0,\varphi_0,\varphi_1,c_1^b,\cdots,c_N^b,z_1,\cdots,z_N)$ which either has the form
\begin{align}\label{l++}
L_0=\int_{\varphi_0}^{\varphi_1}\frac{{\rm d}\varphi}{\sqrt{2(F(\varphi_1)-F(\varphi))}}
\end{align}
if one of the four conditions holds:
\\
(a.1) all $z_j>0$ and $q_0\geq 0$;
\\
(a.2) all $z_j>0$,~$q_0<0$ and $\varphi_0<\varphi_1<\varphi^*$;
\\
(a.3) all $z_j<0$, $q_0>0$ and $\varphi_0<\varphi_1<\varphi^*$;
\\
(a.4) $z_iz_j<0$ for some $i,j$ and $\varphi_0<\varphi_1<\varphi^*$.
\\
or has the form
\begin{align}
L_0=\int_{\varphi_0}^{\varphi_1}\frac{{\rm d}\varphi}{\sqrt{2(F(\varphi_0)-F(\varphi))}},\label{l--}
\end{align}
if one of the four conditions holds:
\\
(b.1) all $z_j<0$ and $q_0\leq 0$;\\
(b.2) all $z_j>0$,~$q_0<0$ and $\varphi^*<\varphi_0<\varphi_1$;\\
(b.3) all $z_j<0$, $q_0>0$ and $\varphi^*<\varphi_0<\varphi_1$;\\
(b.4) $z_iz_j<0$ for some $i,j$ and $\varphi^*<\varphi_0<\varphi_1$.
\begin{rem}
It should be pointed out that $L_0$ is a finite number. For instance, under one of the conditions (a.1)-(a.4), we have $f(\varphi_1)>0$ and
\begin{align*}
\frac{1}{\sqrt{2(F(\varphi_1)-F(\varphi)}}\thicksim \frac{1}{\sqrt{2f(\varphi_1)(\varphi_1-\varphi)}},
~~\textit{as}~\varphi\rightarrow \varphi_1,
\end{align*}
which yield the integral (\ref{l++}) is convergent.
\end{rem}

Due to the rich physical parameters, the statements of the results of BVP with  non-equal surface potential are too long. Consequently, we present these expressions in the subsequent three theorems.


\begin{thm}\label{thnon1}{\bf [Saddle]}
Assume $\varphi_0<\varphi_1$ and either $z_iz_j<0$ for some $i,~j\in\{1,\cdots,N\}$ or $z_iq_0<0$ for all $i\in\{1,\cdots,N\}$. 
The D-BVP (\ref{PB})-(\ref{bc}) has a unique solution $\varphi(x)$.
\begin{itemize}

\item[{\rm(i)}] When $\varphi_0<\varphi_1<\varphi^*$, one has

{\rm(i.1)} If $L\leq L_0$ then the electric potential $\varphi(x)$ is increasing in $x$, the electric field $u(x)=\varphi'(x)$ is decreasing in $x$.


{\rm(i.2)} If $L>L_0$ then the electric potential $\varphi(x)$ is non-monotonic, the electric field $u(x)=\varphi'(x)$ is decreasing in $x$.

\item[\rm(ii)] When $\varphi_1=\varphi^*$,  the electric potential $\varphi(x)$ is increasing in $x$, the electric field $u(x)=\varphi'(x)$ is decreasing in $x$.

\item[\rm(iii)] When $\varphi_0<\varphi^*<\varphi_1$,  the electric potential $\varphi(x)$ is increasing in $x$, the electric field $u(x)=\varphi'(x)$ is first decreasing then increasing  in $x$

\item[\rm(iv)]  When $\varphi_0=\varphi^*$,  the electric potential $\varphi(x)$ is increasing in $x$, the electric field $u(x)=\varphi'(x)$ is increasing  in $x$.


\item[\rm(v)]  When $\varphi_0>\varphi^*$, one has

{\rm(v.1)} If $L\leq L_0$ then the electric potential $\varphi(x)$ is increasing in $x$, the electric field $u(x)=\varphi'(x)$ is increasing in $x$.

{\rm(v.2)} If $L>L_0$ then the electric potential $\varphi(x)$ is non-monotonic and the electric field $u(x)=\varphi'(x)$ is increasing in $x$.
\end{itemize}

\end{thm}

\begin{proof}
(i) We first study the monotone solution of D-BVP (\ref{PB})-(\ref{bc}), represented by the arc AB in Figure \ref{nonequ}. Thanks to the first integral $\mathcal{H}(\varphi,u)$, we have
$$
\frac{u^2(x)}{2}+F(\varphi(x))=\mathcal{H}(\varphi_1,\sqrt{2\beta})=\beta+F(\varphi_1),~~~\beta\geq0.
$$
which implies
$$
\frac{{\rm d}\varphi}{{\rm d}x}=\sqrt{2(\beta+F(\varphi_1)-F(\varphi))},~~\textit{or}~~\frac{{\rm d}\varphi}{\sqrt{2(\beta+F(\varphi_1)-F(\varphi))}}={\rm d}x.
$$
Integrating from $x=-L/2$ to $x=L/2$, one gets the  \emph{time} from $A$ to $B$
\begin{align}\label{time2n}
T_1(\beta):=\int_{\varphi_0}^{\varphi_1}\frac{{\rm d}\varphi}{\sqrt{2(\beta+F(\varphi_1)-F(\varphi))}}=L.
\end{align}
Then the monotone solution of  D-BVP (\ref{PB})-(\ref{bc})  is in one-to-one correspondence with the solutions of (\ref{time2n})
with respect to the variable $\beta\geq0$.
In addition, it is easy to see that $T_1$ is decreasing in $\beta$, and
 $$
 \lim_{\beta \rightarrow\infty}T_1(\beta)=0,~~\lim_{\beta \to 0^+}T_1(\beta)=L_0.
 $$
 To summarize, if $L\leq L_0$, BVP (\ref{PB})-(\ref{bc}) has a unique monotone solution, if $L>L_0$, it has no monotone solutions.

Next we study the non-monotone solution of D-BVP (\ref{PB})-(\ref{bc}), represented by the arc CEF in Figure \ref{nonequ}. We introduce the time-map
\begin{align}\label{time2}\begin{split}
T_2(\alpha):=&T_{21}(\alpha)+T_{22}(\alpha)\\
=&\int_{\varphi_0}^{\alpha}\frac{{\rm d}\varphi}{\sqrt{2(F(\alpha)-F(\varphi))}}+\int_{\varphi_1}^{\alpha}\frac{{\rm d}\varphi}{\sqrt{2(F(\alpha)-F(\varphi))}},~~\alpha\in(\varphi_1,\varphi^*),
\end{split}
\end{align}
where $T_{21}$ denotes the \emph{time} from $C$ to $E$, $T_{22}$ denotes the \emph{time} from $D$ to $E$ which is also
equal to the \emph{time} from $E$ to $F$.  Then the non-monotone solution of  D-BVP (\ref{PB})-(\ref{bc})  is in one-to-one correspondence with the solutions of $T_2(\alpha)=L$
with respect to $\alpha\in(\varphi_1,\varphi^*)$. From the proof of Theorem \ref{thequ} we conclude that
 $$
\lim_{\alpha\rightarrow\varphi_1^+}T_2(\alpha)=L_0,~\lim_{\alpha\rightarrow\varphi^{*-}}T_2(\alpha)=\infty,~~\frac{{\rm d}T_2}{{\rm d}\alpha}>0.
$$
Hence, if $L\leq L_0$, then the D-BVP has no non-monotone solutions, if $L>L_0$, then the D-BVP has a unique non-monotone solution.

Finally, combing with phase portraits and time-maps $T_1(\beta)$ and $T_2(\alpha)$, one can easily obtain properties of $\varphi(x)$, $\varphi'(x)$ and  $\sigma(L,q_0)$.

(ii)-(v) Combing the above proof with similar arguments in the proof of Theorem \ref{thequ}, one can prove them easily.
\end{proof}

Similar to  equal surface potential case in Section \ref{secSCD}, we introduce the magnitude
\begin{align}\label{SCD2}
\sigma(L,q_0):=\max\{|\varphi'(-L/2)|,~|\varphi'(L/2)|\}.
\end{align}
Combing with phase portraits, we have the next direct corollary.

\begin{cor}
Assume $\varphi_0<\varphi_1$ and either $z_iz_j<0$ for some $i,~j\in\{1,\cdots,N\}$ or $z_iq_0<0$ for all $i\in\{1,\cdots,N\}$.
The function $\sigma(L,q_0)$ is increasing in $L$ and satisfies
\[
\lim_{L\rightarrow \infty}\sigma(L,q_0)=\max\left\{\sqrt{2(F_e-F(\varphi_0))},\sqrt{2(F_e-F(\varphi_1))}   \right\}.
\]
\end{cor}

Next, we deal with the D-BVP with non-equal surface potential  in equilibrium-at-infinity case and non-equilibrium case.
Observing the PB equation (\ref{PB}) is invariant with respect to the transformation
$$
(\varphi,z_1,\cdots,z_n,q_0)\rightarrow (-\varphi,-z_1,\cdots,-z_n,-q_0).
$$
Therefore, for simplicity, we only state the results for which all valences $z_i$ are positive in equilibrium-at-infinity case and non-equilibrium case.

\begin{thm}\label{thnon2}{\bf [Equilibrium-at-infinity]}
Assume $\varphi_0<\varphi_1$, $q_0=0$ and $z_j>0$ for $j=1,\cdots,N$. The D-BVP (\ref{PB})-(\ref{bc}) has a unique solution $\varphi(x)$,  the electric field $\varphi'(x)$ is  decreasing in $x$ and
\[\sigma(L)=\max\{|\varphi'(x)|: x\in [-L/2,L/2]\}\]
is increasing in $L$. In addition, one has
\begin{itemize}
 \item[{\rm(i)}] if $L\leq L_0$ then the electric potential $\varphi(x)$ is increasing in $x$ and
\[\lim_{L\rightarrow L_0}\sigma(L)=\sqrt{2(F(\varphi_1)-F(\varphi_0))};\]

\item[{\rm(ii)}] if $L>L_0$ then the electric potential $\varphi(x)$ is non-monotonic and
\[\lim_{L\rightarrow\infty}\sigma(L)=\sqrt{2\left(F_e-F(\varphi_0)\right)}.\]
\end{itemize}
\end{thm}

\begin{thm}\label{thnon3}{\bf [No equilibrium]}
Assume $\varphi_0<\varphi_1$, $q_0>0,~z_j>0$ for  $j=1,\cdots,N$. The D-BVP (\ref{PB})-(\ref{bc}) has a unique solution $\varphi(x)$, the electric field $u(x)=\varphi'(x)$ is decreasing in $x$, and
\[\sigma(L,q_0)=\max\{|\varphi'(x)|: x\in [-L/2,L/2]\}\]
is increasing in $L$. In addition, one has
\begin{itemize}
 \item[{\rm(i)}]  if $L\leq L_0$ then the electric potential $\varphi(x)$ is increasing in $x$ and
\[\lim_{L\rightarrow L_0}\sigma(L,q_0)=\sqrt{2(F(\varphi_1)-F(\varphi_0))};\]
\item[{\rm(ii)}]  if $L>L_0$ then the electric potential $\varphi(x)$ is not monotonic and
\[\lim_{L\rightarrow\infty}\sigma(L,q_0)=\infty.\]
\end{itemize}
\end{thm}




At the end of this section, our focus lies on examining the distinct effects of several physical parameters, namely, the surface potentials $\varphi_0$ and $\varphi_1$, the fixed charge $q_0$, the bulk concentrations $c_j^b$, and the ionic valences $z_j$, on the critical distances $L_0$. For the sake of brevity, we restrict our consideration to the specific form of $L_0$ as given in (\ref{l++}) under the conditions (a.1)-(a.4).

  By performing straightforward calculations, we observe that $L_0$ exhibits a decreasing trend with respect to $q_0$ and $\varphi_0$, while showing an increasing pattern with $\varphi_1$. Moreover, the behavior of $L_0$ is found to decrease with $c_j^b$ when $z_j>0$, and increase with $c_j^b$ when $z_j<0$.
These findings imply that enhancing the left surface potential, fixed charge, and bulk concentrations of cations, or diminishing the right surface potential and bulk concentrations of anions, can effectively reduce the critical distance $L_0$. Consequently, such adjustments increase the likelihood of the transition from monotonic solutions to non-monotonic solutions.

The impact of the ionic valence $z_i$ on the critical distance $L_0$ exhibits a degree of intricacy. After conducting some calculations, we have derived the next findings:

\begin{itemize}

\item[\rm{(i)}]  Assume $z_j>0$.  The higher the valence $z_j$, the shorter the critical distance $L_0$ if either $\varphi_1<0$ or $\varphi_1>0,~z_j<\varphi_1^{-1}$; the higher the valence $z_i$, the longer the critical distance $L_0$ if
    $\varphi_0>0,~z_j>\varphi_0^{-1}$.

 \item[\rm{(ii)}]  Assume $z_j<0$.  The higher the valence $z_i$, the shorter the critical distance $L_0$ if either $\varphi_0>0$ or $\varphi_0<0,~z_j>\varphi_0^{-1}$; the higher the valence $z_j$, the longer the critical distance $L_0$ if
    $\varphi_1<0,~z_j<\varphi_1^{-1}$.

\end{itemize}

\section{Conclusion}\label{concl}

In the present study, we employed a dynamical systems approach to examine the classical PB equation. We provide a more detailed analysis of how the properties of ionic concentrations depend on the topological structures of the PB system and how various system parameters determine important experimental physical quantities. By utilizing a Hamiltonian formulation, we are able to visually observe the distribution of trajectories with intriguing properties in the phase space. Additionally, we have developed explicit formulas for the GC length as in (\ref{gc10})-(\ref{gc20}),  electric pressure as in  (\ref{ep}),
 critical length for monotonicity of electric potential as in (\ref{l++})-(\ref{l--}). Our findings indicate that the presence of equilibrium plays a significant role in predicting the behavior of the PB theory.

Furthermore, we would like to emphasize that the setup in this study is relatively simple, as the classical PB equation assumes point-like ions and neglect ion-ion correlations. In realistic applications, steric effects, especially in ion distribution in aqueous solutions, necessitate the inclusion of local/nonlocal functionals in the electrochemical potential to account for ionic correlations. The analytical framework developed in this study is desirable to extend and apply to investigate other modified PB equations incorporating ionic correlations, which will be addressed in future research.

Lastly, it is worth noting that while PB equations are primarily introduced to model electrolyte solutions under electrochemical equilibrium conditions (i.e., non-flux setups), another crucial electrodiffusion process involves ion transport through membrane channels in physiology. The latter represents a dynamic or non-equilibrium situation, where equilibrium can be viewed as a special case. From a mathematical standpoint, PB models can be regarded as the equilibrium instances of dynamical Poisson-Nernst-Planck  models \cite{pnp2,pnp3,lin2015,liu2013}. Our results should provide valuable insights for further investigation into Poisson-Nernst-Planck  models.



  \end{document}